\documentclass{preprint}

\Title{Lattice Approximation for Stochastic Reaction Diffusion Equations with One-Sided Lipschitz Condition}
\ShortTitle{Lattice Approximation for SRDEs}
\Author{Martin Sauer and Wilhelm Stannat} 
\Address{Institut f\"ur Mathematik, Technische Universit\"at Berlin, Stra\ss e des 17. Juni 136, D-10623 Berlin, Germany and Bernstein Center for Computational Neuroscience, Philippstr. 13, D-10115 Berlin, Germany. E-mail address: \texttt{sauer@math.tu-berlin.de}, \texttt{stannat@math.tu-berlin.de}}
\Abstract{We consider strong convergence of the finite differences approximation in space for stochastic reaction diffusion equations in one dimension with multiplicative noise under a one-sided Lipschitz condition only. The equation may be additionally coupled with a noisy control variable with global Lipschitz condition but no diffusion. We derive convergence with an implicit rate depending on the regularity of the exact solution. This can be made explicit if the variational solution has more than its canonical spatial regularity. As an application, spatially extended FitzHugh-Nagumo systems with noise are considered.}
\AMSsub{Primary 60H15, 60H35, Secondary 35R60, 65C30, 92C20}

\usepackage[format=plain,width=0.5\textwidth,font={footnotesize},labelfont=bf, justification=justified]{caption}
\allowdisplaybreaks

\newcommand{\tv}{\tilde{v}}
\newcommand{\tw}{\tilde{w}}
\newcommand{\vek}[2]{\begin{pmatrix}#1 \\ #2 \end{pmatrix}}
\newcommand{\dzeta}{\, \mathrm{d}\zeta}
\newcommand{\drho}{\,\mathrm{d}\rho}

\begin{document}
\section{Introduction}
Stochastic partial differential equations arise as a model in many scientific fields and the theory of numerical solutions to such equations is growing successfully. However, in applications it is often the case that particular examples such as stochastic reaction diffusion equations, even with very simple structure of the nonlinear term, fail to fit in the assumptions of the majority of publications in this research area. In this article, we study such an SPDE of reaction diffusion type
\begin{equation}\label{eq:SRDE}
\begin{split}
\dot{v}(t,\xi) &= \partial_\xi^2 v(t,\xi) + \phi_1 \big(\xi,v(t,\xi),w(t,\xi)\big) + b_1 \big(\xi,v(t,\xi),w(t,\xi) \big) \eta_1(t,\xi),\\
\dot{w}(t,\xi) &= \phi_2 \big(\xi, v(t,\xi), w(t,\xi) \big) + b_2 \big(\xi,v(t,\xi),w(t,\xi) \big) \eta_2(t,\xi),
\end{split}
\end{equation}
for $t \geq 0$, $\xi \in (0,1)$, where the nonlinear drift term for $v$ satisfies a one-sided Lipschitz condition only. Here, $\eta_j(t,\xi)$ are some Gaussian noise processes to be precisely defined later. This system consists of two variables $v$ and $w$, where the former satisfies some semi-linear stochastic evolution equation and the latter is a kind of control variable coupled to $v$ in an equation without diffusion but global Lipschitz condition on the coefficients. The motivation for such a system comes from the field of neurobiology. In particular, we consider the spatially extended stochastic FitzHugh-Nagumo system, modeling the propagation of the action potential in the axon of a neuron; see e.\,g. \cite{Ermentrout,Keener} and Section 6 for details.

The purpose of this article is to prove strong convergence, i.\,e., convergence in the $p$th mean, including explicit error estimates of a spatial approximation scheme, both easy to implement and widely used in applied sciences. We consider the well-known finite difference method, which has been studied by many authors, but to the best of our knowledge there is no global convergence result with explicit rates if the nonlinear drift term is not Lipschitz continuous. Crucial references in the literature studying similar equations to \eqref{eq:SRDE} (always with $w=0$) are Gy\"ongy \cite{Gyongy99}, where strong convergence with rates and uniform convergence in probability without a rate is proven, assuming a global Lipschitz condition on $\phi_1$ and $b_1$ and no global Lipschitz condition on $\phi_1$ but constant $b_1$, respectively. Also see Pettersson and Signahl \cite{Pettersson05} for a similar result with multiplicative noise. We also want to mention Shardlow \cite{Shardlow99} for strong convergence with rates 
under global Lipschitz condition on $\phi_1$ and additive space-time white noise, as well as Hausenblas \cite{Hausenblas02} for similar results in abstract Banach spaces. A further reference is a series of articles by Gy\"ongy and Millet \cite{GyongyMillet05, GyongyMillet09}, where the authors work in the variational approach and incorporate various spatial approximation schemes for SPDEs driven by a finite dimensional noise. However, equation \eqref{eq:SRDE} is beyond their results since it excludes polynomially growing nonlinearities. On the other hand, it is worth mentioning that using the mathematically more elegant spectral Galerkin approach there are existing results which incorporate our one-sided Lipschitz assumption. We refer to Lord and Rougement \cite{LordRougement}, Jentzen \cite{Jentzen}, and Liu \cite{GinzburgLandau} for strong and pathwise convergence results with explicit rates. Generally speaking, there are only a few references concerning the numerical approximation of nonlinear SPDEs without a global Lipschitz condition and as further examples we include \cite{CarelliProhl} for some local convergence convergence results for 2D stochastic Navier-Stokes equations and \cite{Shardlow05} for applications in the field of neurobiology studying similar equations.

The SPDE \eqref{eq:SRDE} is studied within the variational approach, see e.\,g., \cite{prevot}, and we only use the a priori known regularity of the exact solution to prove strong $p$th mean convergence with an implicitly given rate in Theorem \ref{MainTheorem}. This rate is given in terms of the regularity of the exact solution. If there is better a priori information on the exact solution, this rate can be made explicit in terms of the approximation parameter. The proof is essentially based on uniform exponential a priori estimates for the approximating solutions obtained in Propositions \ref{propAPriori} and \ref{prop:wn} as well as It\^o's formula for the variational solution, in contrast to the popular mild solution approach used in most of the references mentioned above.

Also, most of them concern fully discrete schemes, i.\,e., a discretization of the space and time domain. We only consider the spatial approximation part, which reduces the infinite dimensional problem to an SDE on a finite dimensional space, and then has to be solved with existing theory and the errors accumulate. Nevertheless, let us make a brief comment on the time discretization. One of the most popular approaches is the Euler-Maruyama scheme because of its simplicity and low computational effort. However, on the one hand, the finite difference approximation of \eqref{eq:SRDE} is always a problem of stiff character, thus suggesting the use of time-implicit solvers. On the other hand, considering the nonlinear term for itself, it is known that the Euler-Maruyama scheme converges to the exact solution of the SDE for equations with globally Lipschitz continuous drift and diffusion coefficients. Of course, also our approximated equation fails to satisfy such a condition. For such equations with super-linearly growing drift and/or diffusion it was shown recently by Hutzenthaler, Jentzen, and Kloeden \cite{HJK2011} that the Euler-Maruyama approximations diverge in a strong and weak sense. An obvious remedy for this problem is to use, e.\,g., the implicit Euler scheme already suggested for the linear term, which is known to converge; however, it requires more computational effort than the explicit scheme. Again Hutzenthaler, Jentzen, and Kloeden \cite{HJK2012} suggest a so-called tamed Euler scheme that is explicit and computationally less expensive. It is also observed that the runtime as a function of the dimension grows linearly and quadratic for the tamed Euler and implicit Euler scheme, respectively. Of course, this favors the tamed Euler scheme for our high dimensional approximating problem. Thus, we propose solving \eqref{eq:SRDE} with a combination of finite differences and a semi-implicit tamed Euler scheme.

The article is structured as follows. In the next section, we describe the precise setting and assumptions on the coefficients of \eqref{eq:SRDE} and state the existence and uniqueness result for its corresponding abstract stochastic evolution equation in the sense of variational solutions. In Section 3 we introduce the approximation scheme and state the main result in Theorem \ref{MainTheorem}. Sections 4 and 5 contain the proof of this theorem. As an application of our results we consider in Section 6 a spatially extended FitzHugh-Nagumo system with noise studied by Tuckwell \cite{TuckwellFHN} from a more applied point of view. In this article, the impact of noise on the generation of action potentials and the reliability of faithful signal transmission was studied. The obtained results depend on the chosen numerical scheme and Theorem \ref{MainTheorem} now allows us to explicitly quantify the approximation error of the finite difference approximation used therein. As an illustration for the use of our result,
 we state an estimator for the probability of propagation failure based on numerical observations as well as confidence intervals for this estimator. These depend on the statistical Monte-Carlo and the numerical approximation error and are explicitly constructed.

\section{Mathematical Setting and Assumptions}

Let $(H, \norm{\cdot}_H) = L^2(0,1)$ with scalar product denoted by $\scp{\cdot}{\cdot}_H$ and consider the Laplacian as a linear operator on $H$, $A: D(A) \subset H \to H$, i.\,e.,
\[
\big(A v\big)(\xi) \df \partial_\xi^2 v(\xi)
\]
equipped with (homogeneous) Neumann boundary conditions in $0$ and $1$. Corresponding to this, recall the definition of the fractional Sobolev spaces $H^\theta(0,1)$ as $D\big( (-A)^{\frac{\theta}{2}} \big)$ or for $\theta > 0$ with the equivalent Sobolev-Slobodeckij norm (cf. \cite{TriebelAlt})
\[
\norm{u}_{H^\theta} \df \sum_{s=0}^{\lfloor \theta \rfloor} \norm{\partial_\xi^s u}_H + \Bigg( \iint_0^1 \frac{\abs{\partial_\xi^{\lfloor \theta \rfloor} u(\xi) - \partial_\xi^{\lfloor \theta \rfloor} u(\zeta)}^2}{\abs{\xi-\zeta}^{2(\theta - \lfloor \theta \rfloor)+1}} \dzeta \dxi \Bigg)^{\frac12}.
\]
Define, in particular, the space $(V, \norm{\cdot}_V)$ as $H^1$ with norm $\norm{u}_V^2 \df \int_0^1 \abs{\partial_\xi u(\xi)}^2 \dxi$ and denote by $V^\ast$ its topological dual, thus we have the Gelfand triplet $V \hookrightarrow H \hookrightarrow V^\ast$ with continuous and dense embeddings. As auxiliary spaces we define the product spaces $\mathcal{V} \df V \times H$, $\mathcal{H} \df H \times H$ with $\mathcal{V} \hookrightarrow \mathcal{H} \hookrightarrow \mathcal{V}^\ast$ continuously and densely. In the following, we define the abstract drift and diffusion operators mapping $\mathcal{V}$ to $\mathcal{V}^\ast$.

It is well known that $A$ can be uniquely extended to $A: V \to V^\ast$ via an integration by parts
\[
\dualp{Au}{v} = \int_0^1 \partial^2_\xi u(\xi) v(\xi) \dxi = - \int_0^1 \partial_\xi u(\xi) \partial_\xi v(\xi) \dxi \leq \norm{u}_V \norm{v}_V.
\]
The functions $\phi_1, \phi_2: [0,1] \times \R \times \R \to \R$ are continuous and satisfy further conditions specified below in detail. More or less, $\phi_1$ is one-sided Lipschitz in the variable $v$, the dependence on the control variable $w$ as well as $\phi_2$ are assumed to be globally Lipschitz. Moreover, some recurrent behavior is required.
\begin{assum}\label{assum:Nonlinear}
There exist constants $L, \beta, \gamma > 0$ and $1 < m \leq 3$ such that
\begin{align}
&2\vek{\phi_1(\xi_1,v_1,w_2) - \phi_1(\xi_2,v_2,w_2)}{\phi_2(\xi_1,v_1,w_1) - \phi_2(\xi_2,v_2,w_2)} {\kern -0.3em} \vek{v_1-v_2}{w_1-w_2} \leq L \big( \abs{v_1-v_2}^2 + \abs{w_1-w_2}^2 \big), \tag{A1}\label{A1}\\
&2\vek{\phi_1(\xi,v,w)}{\phi_2(\xi,v,w)}{\kern -0.3em}\vek{v}{w} \leq - \beta \abs{w}^2 + L \abs{v}^2 - \gamma \abs{v}^{m+1}, \tag{A2}\label{A2}\\
\intertext{and, moreover,}
\begin{split}
&{\kern 0.5em}\abs{\phi_1 \big(\xi_1,v_1,w_1 \big) - \phi_1 \big(\xi_2,v_2,w_2\big)} \leq  L \abs{w_1-w_2}\\
&\qquad+  L \big( 1 + \abs{v_1}^{m-1} + \abs{v_2}^{m-1} \big)\big( \abs{\xi_1-\xi_2} + \abs{v_1-v_2}\big),
\end{split}\tag{A3}\label{A3}\\
&{\kern 0.5em}\abs{\phi_2 \big(\xi_1,v_1,w_1 \big) - \phi_2 \big(\xi_2,v_2,w_2\big)} \leq L \big( \abs{\xi_1 - \xi_2} + \abs{v_1-v_2} + \abs{w_1-w_2} \big),\tag{A4}\label{A4}
\end{align}
for all $\xi,\xi_1,\xi_2 \in [0,1]$, $v,v_1,v_2, w,w_1,w_2 \in \R$.
\end{assum}
A typical form of $\phi_1$ is a polynomial in $v$ of odd degree and negative leading coefficient together with a linear perturbation in $w$, for example, $\phi_1(v,w) = v - v^3 - w$ similarly as in the FitzHugh-Nagumo system described in Section 6. With these assumptions we can define the Nemytskii operator $\Phi: \mathcal{V} \to \mathcal{V}^\ast$ by
\[
\big(\Phi(v,w)\big)(\xi) \df \vek{\phi_1\big(\xi, v(\xi),w(\xi)\big)}{\phi_2\big(\xi, v(\xi),w(\xi)\big)}
\]
and estimate for $v_1,v_2 \in V$, $w_1,w_2 \in H$ with the embedding $V \hookrightarrow L^\infty(0,1)$ as follows:
\[
\phantom{\,}_{\mathcal{V}^\ast} \scp{\Phi(v_1,w_1)}{(v_2,w_2)}_{\mathcal{V}} \leq C \Big( \norm{v_1}_V^m \norm{v_2}_V + \big(1 + \norm{w_1}_H + \norm{v_1}_H\big) \big(\norm{v_2}_H +\norm{w_2}_H\big) \Big).
\]
Now that the assumptions on the drift of equation \eqref{eq:SRDE} are stated, let us focus on the noise process $\eta = (\eta_1, \eta_2)$, formally given as the derivative of some $Q_j$-Wiener processes $\sqrt{Q_j} W_j(t)$, where $(W_j(t))_{t \geq 0}$, $j=1,2$ are two independent cylindrical Wiener processes on $H$ with respect to an underlying probability space $(\Omega, \algF, \algF_t, \PP )$. In order to obtain a solution to \eqref{eq:SRDE} we cannot treat the case of space-time white noise in neither variable, i.\,e., $Q_j=I$, but only the case of colored noise with nuclear covariance operators. This is inevitable in the variational framework used herein (instead of the also widely used mild formulation) and gives us It\^o's formula for the square of the $\mathcal{H}$-norm as a major tool. The assumptions on $Q_j$ and $b_j$ are combined below.
\begin{assum}\label{assum:Noise}
Let $\theta_1 > 1/2$, $\theta_2 \geq 1$ and $Q_j \in L(H)$, $j=1,2$ be symmetric and positive definite. Assume that $\tr (-A)^{\theta_j} Q_j < \infty$, in particular, $Q_j$ admits an integral kernel of the form
\begin{equation}
\big(Q_j u\big)(\xi) = \iint_0^1  q_j(\xi,\zeta) q_j(\rho,\zeta) u(\rho) \dzeta \drho, \quad q_j \in H^{\theta_j} \big((0,1)^2\big). \tag{A5}\label{A5}
\end{equation}
In the case of $Q_2$ we assume furthermore that $q_2 \in L^\infty((0,1)^2)$. The scalar valued noise intensities $b_j: [0,1] \times \R \times \R \to \R$ are supposed to be bounded and Lipschitz continuous, in particular, $\abs{b_j(\xi,v,w)} \leq 1$ and
\begin{equation}
\abs{b_j(\xi_1,v_1,w_1) - b_j(\xi_2,v_2,w_2)} \leq \big( \abs{\xi_1-\xi_2} + \abs{v_1-v_2} + \abs{w_1-w_2}\big) \tag{A6}\label{A6}
\end{equation}
for all $\xi,\xi_1,\xi_2 \in [0,1]$ and $v,v_1,v_2, w, w_1,w_2 \in \R$.
\end{assum}
\begin{rem}
The stronger assumptions for $Q_2$ are due to the less regularizing drift of $w$ and are necessary for the derivation of sufficiently strong a priori estimates of the approximating solutions.
\end{rem}
In the same manner as for the nonlinear drift, we define the Nemytskii operators for the noise as $B_j(v,w)(\xi) \df b_j (\xi,v(\xi),w(\xi))$, $j=1,2$. Furthermore, let $x = (v,w)^T \in \mathcal{V}$ and
\[
\mathcal{A}(x) \df \vek{Av}{0}+ \Phi (v,w), \quad \mathcal{B}(x) \df \vek{B_1(v,w) \sqrt{Q_1}}{B_2(v,w) \sqrt{Q_2}}, \quad W(t) \df \vek{W_1(t)}{W_2(t)}.
\]
Then, the stochastic evolution equation for $X(t) \df \big(v(t), w(t)\big)^T$ corresponding to equation \eqref{eq:SRDE} is given by
\begin{equation}\label{eq:EvoEq}
\mathrm{d} X(t) = \mathcal{A}\big(X(t)\big) \dt + \mathcal{B}\big(X(t)\big) \dwt, \quad X(0) = x_0.
\end{equation}
Existence and uniqueness of a solution for such an equation can be studied in the variational framework (see for example \cite{prevot}), in particular, a recent extension in \cite{LiuLocallyMonotone}. The following theorem is a corollary to \cite[Theorem 1.1]{LiuLocallyMonotone}.
\begin{thm}\label{thm:Existenz}
Suppose $T>0$, $x_0 \in \mathcal{L}^{p}(\Omega, \algF, \PP;H)$, $p \geq 2m$, and Assumptions \ref{assum:Nonlinear} and \ref{assum:Noise} are satisfied. Then, equation \eqref{eq:EvoEq} has a unique variational solution $(X(t))_{t \in [0,T]}$ which satisfies
\[
\EV{\sup_{t \in [0,T]} \norm{X(t)}_\mathcal{H}^p + \int_0^T \norm{X(t)}_\mathcal{V}^2 \dt} < \infty.
\] 
\end{thm}
\begin{proof}
We have to verify the conditions (H1)--(H4) in \cite{LiuLocallyMonotone}. Of course, for $x_1,x_2,x_3 \in \mathcal{V}$ the map $s \mapsto \phantom{\,}_{\mathcal{V}^\ast} \scp{\mathcal{A}(x_1+sx_2)}{x_3}_\mathcal{V}$ is continuous on $\R$, hence (H1) holds. Furthermore, we can easily obtain the monotonicity of $\mathcal{A}$ by the one-sided Lipschitz condition \eqref{A1}
\[
2 \phantom{\,}_{\mathcal{V}^\ast} \scp{\mathcal{A}(x_1) - \mathcal{A}(x_2)}{x_1-x_2}_\mathcal{V} \leq -2 \norm{v_1-v_2} + 2 L \norm{x_1-x_2}_\mathcal{H}^2.
\]
Moreover, $\mathcal{B}$ is Lipschitz continuous in the Hilbert-Schmidt norm $\norm{\cdot}_{L_2}$ since
\begin{equation}\label{LipschitzB}
\begin{split}
\norm{\mathcal{B}(x_1) - \mathcal{B}(x_2)}_{L_2(\mathcal{H})} &\leq \sum_j \norm{B_j(v_1,w_1) - B_j(v_2,w_2)}_{L(H^{\theta_j},H)} \norm{\sqrt{Q_j}}_{L_2(H,H^{\theta_j})}\\
& \leq c_{\theta_j} \norm{x_1-x_2}_{\mathcal{H}} \sum_j \big(\tr (-A)^{\theta_j} Q_j\big)^{\frac12}.
\end{split}
\end{equation}
The latter inequality holds because of the embedding $H^\theta(0,1) \hookrightarrow L^\infty(0,1)$ for $\theta > \frac12$ with constant $c_\theta$. Thus, we have verified (H2) with $\rho \equiv 0$. \eqref{A2} together with the fact that $\norm{B(x)}_{L(H)} \leq 1$ implies
\[
2 \phantom{\,}_{\mathcal{V}^\ast}\scp{\mathcal{A}(x)}{x}_\mathcal{V} + \norm{\mathcal{B}(x)}_{HS}^2 \leq -2 \norm{x}_\mathcal{V}^2 + (L+2) \norm{x}_\mathcal{H}^2 + \tr Q
\]
for all $x \in \mathcal{V}$ and (H2) holds with $\alpha=2$. With \eqref{A3} and \eqref{A4} the matching growth condition is
\begin{align*}
&\norm{\mathcal{A}(x)}_{\mathcal{V}^\ast} \leq \norm{x}_\mathcal{V} + {\kern -0.7em} \sup_{\norm{u}_V=1} \int_0^1 \abs{\phi_1 (v,w) u}(\xi) \dxi + {\kern -0.7em}\sup_{\norm{u}_H=1} \int_0^1 \abs{\phi_2 ( v,w) u}(\xi) \dxi\\
&\quad\leq \norm{x}_\mathcal{V} + C {\kern -0.7em}\sup_{\norm{u}_V=1} \int_0^1 \abs{v(\xi)}^m \abs{u(\xi)}\dxi + C {\kern -0.7em}\sup_{\norm{u}_H=1} \int_0^1 \big(1 + \abs{w(\xi)} + \abs{v(\xi)} \big) \abs{u(\xi)}\dxi \\
&\quad\leq \norm{x}_\mathcal{V}^2 + c_1 C \norm{v}_V \Big(\int_0^1 \abs{v(\xi)}^{2m-2} \dxi\Big)^\frac12 + C \big(1 + \norm{w}_H + \norm{v}_H \big),
\end{align*}
where we used the embedding $V \hookrightarrow L^\infty(0,1)$. The latter is finite if $m \leq 3$, hence
\[
\norm{\mathcal{A}(x)}_{\mathcal{V}^\ast}^2 \leq C \big(1 + \norm{x}_\mathcal{V}^2 \big) \big( 1 + \norm{x}_\mathcal{H}^{2m-2}\big)
\]
and condition (H4) in \cite{LiuLocallyMonotone} holds. This condition determines the lower bound $p\geq (2m-2) + 2 = 2m$ for the integrability of the initial condition $x_0$.
\end{proof}

\section{The Finite Differences Scheme and the Main Result}

In this section we briefly describe the approximation scheme and state the main convergence result in Theorem \ref{MainTheorem}. Equation \eqref{eq:EvoEq} is spatially approximated using an equidistant grid given by $\tfrac1n \{ 0, \dots, n \}$, approximating the domain $(0,1)$, and the vectors $v^n, w^n \in \R^{n+1}$ denote the functions $v$ and $w$ evaluated on this grid. Furthermore, we introduce the spaces $V_n \cong \R^{n+1}$ with the discrete boundary conditions $u_1-u_0 = u_n - u_{n-1} = 0$ corresponding to the homogeneous Neumann boundary. We equip $V_n$ with the semi-norm $\norm{v}_n^2 \df n \sum_{k=1}^n (v_k - v_{k-1})^2$. The finite difference approximation $A^n$ of $A$ is just the discrete Laplacian, given by
\[
(A^n v)_k \df n^2 (v_{k+1} - 2 v_k + v_{k-1}), \quad 1 \leq k \leq n-1, \quad v \in V_n
\]
together with the appropriate boundary values $(A^n v)_0 = n^2(v_1-v_0)$, $(A^nv)_n = n^2(v_n-v_{n-1})$ which are $0$ for $v \in V^n$. This allows to imitate the variational approach in the discrete setting since a summation by parts formula holds, i.\,e.
\begin{equation}\label{SBP}
\tfrac{1}{n} \sum_{k=0}^n \big(A^n v\big)_k u_k = - n \sum_{k=1}^n (v_k - v_{k-1}) ( u_k - u_{k-1})
\end{equation}
for all $v \in V^n$, $u \in \R^{n+1}$.  The nonlinearities are simply evaluated pointwise, i.\,e.,
\[
\phi^n_j(v,w) \df \Big( \phi_j \big(\tfrac{k}{n}, v_k, w_k \big)\Big)_{0 \leq k \leq n} \quad\text{and}\quad b_j^n(v) \df \operatorname{diag} \Big( b_j \big(\tfrac{k}{n}, v_k, w_k \big)\Big)_{0 \leq k \leq n}
\]
for $v,w \in V_n$, $j=1,2$ where by $\operatorname{diag} v$ we denote the diagonal matrix with the entries of the vector $v$ on the main diagonal. In the next step let us construct the approximating noise in terms of the given realization of the driving cylindrical Wiener processes $(W_j(t))_{t \geq 0}$, $j=1,2$. Recall that
\begin{equation}\label{eq:BM}
\scp{W_j(t)}{\sqrt{n} 1_{[\frac{k-1}{n}, \frac{k}{n})}}_H \fd \beta_{j,k}(t), \quad 1 \leq k \leq n
\end{equation}
defines a family of $2n$ iid real valued Brownian motions. The spatial covariance structure given by the kernels $q_j$ is discretized as
\[
q^n_{j,k,l} \df n^2 \int_{\frac{k-1}{n}}^{\frac{k}{n}}\int_{\frac{l-1}{n}}^{\frac{l}{n}}q_j (\xi, \zeta) \dzeta \dxi, \quad 1 \leq k,l \leq n,
\]
together with $q^n_{j,0,l} \df 0$, $1\leq l \leq n$. With this discrete covariance matrix we can replace $q$ by a piecewise constant kernel, i.\,e., for $ 1 \leq k \leq n$ and $\xi \in \big[\frac{k-1}{n},\frac{k}{n}\big)$,
\begin{align*}
&\int_0^1 q_j (\xi, \zeta) W_j(t,\zeta) \dzeta \;\rightsquigarrow \; \sum_{l=1}^n q^n_{j,k,l} \int_0^1 1_{[\frac{l-1}{n}, \frac{l}{n})}(\zeta) W_j(t,\zeta) \dzeta\\
&\quad=\tfrac{1}{\sqrt{n}} \sum_{l=1}^n q^n_{j,k,l} \scp{W_j(t)}{\sqrt{n} 1_{[\frac{l-1}{n}, \frac{l}{n})}} = \tfrac{1}{\sqrt{n}} \sum_{l=1}^n q^n_{j,k,l} \beta_{j,l}(t) \fd \Big(\sqrt{Q_j^n} P_n W_j(t)\Big)_k,
\end{align*}
where $P_n h = (\scp{h}{\sqrt{n}1_{[\frac{l-1}{n}, \frac{l}{n})}})_{1 \leq l \leq n}$. Denote by $W_j^n(t) = P_n W_j(t)$ the resulting $n$-dimensional Brownian motions and the finite dimensional system of stochastic differential equations approximating equation \eqref{eq:EvoEq} or rather \eqref{eq:SRDE} is
\begin{equation}\label{eq:Approximation}
\begin{split}
\mathrm{d} v^n(t) &= \Big[ A^n v^n(t) + \phi_1^n \big(v^n(t),w^n(t)\big) \Big] \dt + b_1^n\big(v^n(t), w^n(t)\big) \sqrt{Q_1^n} \, \mathrm{d} W_1^n(t),\\
\mathrm{d} w^n(t) &= \Big[ \phi^n_2 \big(v^n(t),w^n(t)\big) \Big] \dt + b_2^n\big(v^n(t), w^n(t)\big) \sqrt{Q_2^n} \, \mathrm{d} W_2^n(t).
\end{split}
\end{equation}
Standard results on stochastic differential equations imply the existence of a unique strong solution $(v^n(t),w^n(t))$ to \eqref{eq:Approximation}; see e.\,g., \cite[Chapter 3]{prevot}. In order to formulate the convergence result, we embed the processes $v^n$ and $w^n$ into the space $C([0,T],V)$ by linear interpolation with respect to the space variable $\xi$. Given $v \in V_n$, define
\[
\tv (\xi) \df (n\xi - k + 1) v_k + (k -n\xi) v_{k-1}, \quad \xi \in \big[ \tfrac{k-1}{n},\tfrac{k}{n} \big].
\]
Denote by $\iota_n: V_n \to V; v \mapsto \tv$ the embedding given above. Furthermore, define
\[
\tilde{X}^n(t) \df \vek{\tv^n(t)}{\tw^n(t)} \quad \text{and} \quad E^n(t) \df X(t) - \tilde{X}^n(t).
\]
\vspace{-\baselineskip}
\begin{thm}\label{MainTheorem}
Suppose Assumptions \ref{assum:Nonlinear} and \ref{assum:Noise} hold, $v_0 \in \mathcal{L}^\infty (\Omega,\algF,\PP; C([0,1]))$ and $w_0 \in \mathcal{L}^\infty (\Omega,\algF,\PP; C^1([0,1]))$. Then for every $1\leq p < \infty$ there exists a finite constant $C=C(v_0,w_0,p,T,Q,L,\beta, \gamma)$ such that
\begin{align*}
\EV{\sup_{t \in [0,T]} \norm{E^n(t)}_\mathcal{H}^p}^{\frac{1}{p}} &\leq C\Bigg( \EV{\norm{E^n(0)}_{\mathcal{H}}^{2p}} + \EV{ \Big( \int_0^T I_n\big(v(t)\big)^2 \dt\Big)^p}^\frac12 \Bigg)^{\frac{1}{2p}}\\
\intertext{which converges to $0$ as $n \to \infty$. Denote by $v'(t,\xi) = \partial_\xi v(t,\xi)$ the first derivative with respect to $\xi$, then the rate of convergence is implicitly given by the expression}
I_n(v(t)) &\df \Big( \sum_{k=1}^n \int_{\frac{k-1}{n}}^{\frac{k}{n}} \Big(v'(t,\xi) -  v' \big(t,\xi \pm \tfrac{1}{n}\big) \Big)^2 \dxi \Big)^\frac12.\\
\intertext{If, in addition, $v \in X_\alpha \df \mathcal{L}^{p^\ast}( \Omega, \algF, \PP; L^2([0,T], H^{1+\alpha}))$ for some $\alpha > 0$ and $1<p^\ast<\infty$, then the rate of convergence is explicitly given by}
\EV{ \sup_{t \in [0,T]} \norm{E^n(t)}_{\mathcal{H}}^p}^\frac{1}{p} &\leq C_{\ref{MainTheorem}} \Bigg(\EV{\norm{E^n(0)}_{\mathcal{H}}^{2p}}^{\frac{1}{2p}} + n^{-\frac12 \min \{1, \alpha\}}\Bigg)
\end{align*}
for all $p < p^\ast$ and some constant $C_{\ref{MainTheorem}} = C(v_0,w_0,p,p^\ast,T,Q,L,\beta,\gamma, \norm{v}_{X_\alpha})$.
\end{thm}
\begin{rem}
\begin{enumerate}
\item Existence and uniqueness for solutions to \eqref{eq:EvoEq}, in particular, the verification of the conditions in the variational framework, relies on Sobolev embeddings. Hence generalizations to dimension $d \geq 2$ are more difficult and especially $m=m(d)$ is a decreasing function of $d$. In particular, $m(2)=m(3) = 7/3$; see \cite[Example 3.2]{LiuLocallyMonotone}. For example, the FitzHugh-Nagumo system studied in Section 6 as a relevant neurobiological application is currently not covered in $d\geq 2$.
\item The calculus for the error estimation essentially stays the same if we a priori assume the existence of analytically strong solution with It\^o's formula for $\norm{X(t)}_{\mathcal{H}}^2$. Thus, a separation of existence and uniqueness of solutions and the numerical approximation may be appropriate to study the multivariate case. Also, note that the condition $m \leq 3$ used in the proof of Theorem \ref{MainTheorem} in Step 3 is dimension independent.
\item For the error analysis in the multivariate case it is crucial how the approximation of the linear part and the covariance operator generalize. However, this is yet to be done.
\end{enumerate}
\end{rem}

\section{Uniform A Priori Estimates}
In addition to the a priori estimates in Theorem \ref{thm:Existenz}, the proof of Theorem \ref{MainTheorem} requires such estimates corresponding to the approximating solutions $\tv^n$ and $\tw^n$. These estimates are formally obtained by applying It\^o's formula to the square of the $\mathcal{H}$-norm and we will do the same in the discrete case and obtain an exponential a priori estimate for the $L^2([0,T],V)$-norm of $\tv^n$ and the discrete $l_{m+1}$-norm of $v^n$ uniform in the parameter $n$. Moreover, we derive uniform a priori estimates for $\tw^n$ in $\mathcal{L}^p(\Omega, \algF, \PP; L^\infty([0,T],V))$ for any $1\leq p < \infty$. For a concise statement of the results let us introduce the notation
\[
l_p^n(v) \df \tfrac1n \sum_{k=1}^n \abs{v_k}^p, \quad v \in V_n,\quad 1 \leq p < \infty.
\]
\begin{lem}\label{lem:EnergyEstimate}
Let $n \in \N$. The approximating solution $(v^n(t), w^n(t) )$ satisfies the estimate
\begin{equation}\label{eq:EnergyEstimate}
\begin{split}
&\EV{ l_2^n\big(v^n(t)\big) + l_2^n\big(w^n(t)\big) + 2 \int_0^t \norm{v^n(t)}_n^2 \dt + \gamma \int_0^t l_{m+1}^n\big(v^n(t)\big) \dt}\\
&\quad\leq e^{Lt} \Big( \EV{l_2^n\big(v^n(0)\big) + l_2^n\big(w^n(0)\big)} + t \sum_j \tr Q_j \Big)
\end{split}
\end{equation}
for all $t\in[0,T]$.
\end{lem}
\begin{proof}
It\^o's formula applied to $(v_k^n)^2 + (w_k^n)^2$, the summation by parts formula \eqref{SBP} and the dissipativity condition on $\Phi$ in \eqref{A2} imply
\begin{align*}
&l_2^n\big(v^n(t)\big) + l_2^n\big(w^n(t)\big) + 2 \int_0^t \norm{v^n(s)}_n^2 \ds + \gamma \int_0^t l_{m+1}^n\big(v^n(s)\big)\ds\\
&\quad\leq l_2^n\big(v^n(0)\big) + l_2^n\big(w^n(0)\big) + L \int_0^t l_2^n\big(v^n(s)\big) + l_2^n\big(w^n(s)\big) \ds + M^n_1(t) + M^n_2(t)\\
&\qquad+ \frac{1}{n^2} \int_0^t \sum_{k,l=1}^n \abs{b_1\big(\tfrac{k}{n},v_k^n(s), w_k^n(s)\big) q^n_{1,k,l}}^2 + \abs{b_2\big(\tfrac{k}{n},v_k^n(s),w_k^n(s)\big) q^n_{2,k,l}}^2 \ds.
\end{align*}
Here, $M^n_i(t)$ denotes the stochastic integrals that vanish after taking the expectation using a standard localization argument. Since both $b_1$ and $b_2$ are bounded we obtain
\begin{align*}
&\EV{l_2^n\big(v^n(t)\big) + l_2^n\big(w^n(t)\big) + 2 \int_0^t \norm{v^n(s)}_n^2 \ds + \gamma \int_0^t l_{m+1}^n\big(v^n(s)\big)\ds}\\
&\leq \EV{l_2^n\big(v^n(0)\big) + l_2^n\big(w^n(0)\big)} + L {\kern -0.2em}\int_0^t{\kern -0.3em} \EV{ l_2^n\big(v^n(s)\big) + l_2^n\big(w^n(s)\big)} {\kern -0.2em}\mathrm{d}s + t \sum_j \tr Q_j.
\end{align*}
Gronwall's lemma now yields the result.
\end{proof}
Using a similar strategy we can also obtain exponential a priori estimates that are crucial for a Gronwall type argument in the error analysis. The proof is based on a standard decomposition of the exponential of a martingale into a supermartingale and the exponential of its quadratic variation process. The precise statement is the following.
\begin{prop}\label{propAPriori}
Let $v_0,w_0 \in \mathcal{L}^{\infty}( \Omega, \algF, \PP; C([0,1]))$ and $0< \alpha \leq \alpha^\ast \df \beta/(144 \tr Q_2)$. Then
\[
\EV{\exp \Bigg( \alpha \int_0^T \norm{\tv^n(t)}_V^2 + \gamma l_{m+1}^n\big(v^n(t)\big) \dt \Bigg)} \leq C_{\ref{propAPriori}}
\]
uniformly in $n \in \N$, where $C_{\ref{propAPriori}} = C(\alpha,v_0,w_0,T,Q_1,Q_2,L, m,\gamma )$ is an explicitly known finite constant.
\end{prop}
\begin{proof}
Note, that Lemma \ref{lem:EnergyEstimate} is not optimal in terms of $w^n$ since the recurrent term $-\beta \abs{w}^2$ in \eqref{A2} was not used. However, this is essential for the exponential moments. Let $\alpha >0$, then It\^o's formula as in Lemma \ref{lem:EnergyEstimate} implies
\begin{align}
&\frac{\alpha}{3} \int_0^T \beta\, l_2^n \big(w^n(t)\big) + 2 \norm{v^n(t)}_n^2 + \gamma\, l_{m+1}^n\big(v^n(t)\big) \dt\nonumber\\
\begin{split}
&\quad\leq \frac{\alpha}{3} \Big( l_2^n\big(v^n(0)\big) + l_2^n\big(w^n(0)\big) + T \big(\tr Q_1 + \tr Q_2\big) \Big)\\
&\qquad + \frac{\alpha}{3} L \int_0^T l_2^n \big(v^n(t)\big)\dt + \frac{\alpha}{3} M^n_1(T) + \frac{\alpha}{3} M^n_2(T),
\end{split}\nonumber\\
\begin{split}
&\quad\leq \frac{\alpha}{3} \Big( l_2^n\big(v^n(0)\big) + l_2^n\big(w^n(0)\big) + T \big(\tr Q_1 + \tr Q_2 + C(L,m,\gamma)\big) \Big)\\
&\qquad + \frac{\alpha \gamma}{6} \int_0^T l_{m+1}^n\big(v^n(t)\big) \dt + \frac{\alpha}{3} M^n_1(T) + \frac{\alpha}{3} M^n_2(T),
\end{split} \label{eqproof1}
\end{align}
where again $M^n_j(T)$ denote the stochastic integrals. Instead of using Gronwall's lemma, we absorbed the integral on the right-hand side by the one on the left via Young's inequality, i.\,e., we exchange the exponentially growing (in time) multiplicative constant for an additive one, namely $C(L,m,\gamma)T$. Now recall that for every continuous local martingale $M_t$ vanishing at $t=0$ the process
\[
Z^\alpha_t \df \exp \Big( \alpha M_t - \frac{\alpha^2}{2} \langle M\rangle_t \Big), \quad \alpha > 0,
\]
is again a continuous local martingale, thus a supermartingale by Fatou's lemma and therefore $\EV{Z^\alpha_t} \leq 1$ for all $t$. With this information we can derive
\begin{equation}\label{eqproof2}
\EV{\exp \big(\alpha M_t \big)} \leq \EV{\exp \big(2\alpha^2 \langle M \rangle_t\big)}^{\frac12}.
\end{equation}
After taking expectations in \eqref{eqproof1} and an application of H\"older's inequality, \eqref{eqproof2} implies
\begin{align*}
&\EV{\exp \Bigg( \frac{\alpha}{3} \int_0^T \beta\, l_2^n \big(w^n(t)\big) + 2 \norm{v^n(t)}_n^2 + \frac{\gamma}{2}\, l_{m+1}^n\big(v^n(t)\big) \dt\Bigg)}^3\\
&\quad\leq \EV{\exp \Bigg(\alpha \Big( l_2^n\big(v^n(0)\big) + l_2^n\big(w^n(0)\big) + T \big(\tr Q_1 + \tr Q_2 + C(L,m,\gamma)\big)\Big)\Bigg)}\\
&\qquad \times \EV{\exp \Bigg( \alpha M^n_1(T)\Bigg)} \EV{ \exp\Bigg( \alpha M^n_2(T)\Bigg)}\\
&\quad\leq \EV{\exp \Bigg(\alpha \Big( l_2^n\big(v^n(0)\big) + l_2^n\big(w^n(0)\big) + T \big(\tr Q_1 + \tr Q_2 + C(L,m,\gamma)\big)\Big)\Bigg)}\\
&\qquad \times \EV{\exp \Bigg( 2\alpha^2 \langle M^n_1(\cdot)\rangle_T \Bigg)}^\frac12 \EV{ \exp\Bigg( 2\alpha^2 \langle M^n_2(\cdot)\rangle_T \Bigg)}^\frac12.\\
\end{align*}
In order to calculate the quadratic variations, let us state the explicit formulas for $M^n_j$. These are
\[
M^n_j(t) = \frac{2}{n} \int_0^t \sum_{k,l=1}^n u_{j,k}(s) b_j\big( \tfrac{k}{n}, v^n_k(s), w_k^n(s)\big) \frac{q^n_{j,k,l}}{\sqrt{n}} \,\mathrm{d}\beta_{j,l}, 
\]
where $u_{j,k}(s)$ stands for $v^n_k(s)$ and $w^n_k(s)$ in the cases $j=1$ and $2$, respectively. Their quadratic variations can be bounded by the integrals on the left-hand side, in particular, by
\begin{align*}
\langle M^n_j(\cdot) \rangle_T &= \frac{4}{n^3} \int_0^T \sum_{l=1}^n \Big(\sum_{k=1}^n u_{j,k}(t) b_j\big( \tfrac{k}{n}, v_k^n(t) \big) q^n_{j,k,l} \Big)^2 \dt\\
&\leq \frac{4}{n^2} \Big(\sum_{k,l=1}^n \big(q^n_{j,k,l}\big)^2 \Big) \int_0^T \frac{1}{n}\sum_{k=1}^n \big(u_{j,k}(t)\big)^2 \dt\\
&\leq 4 \tr Q_j \int_0^T l_2^n \big(u_j(t)\big) \dt.
\end{align*}
In the case $j=1$, Young's inequality allows us to absorb this factor by the left-hand side in the same way as in \eqref{eqproof1} in exchange for an additional constant $e^{CT}$ on the right-hand side. In particular, this can be done independent of the size of $\tr Q_1$. This is obviously not the case for $j=2$, hence $\alpha \leq \beta/(24 \tr Q_2)$ should be satisfied. We have shown that
\begin{equation}\label{eqproof3}
\begin{split}
&\EV{\exp \Bigg( \alpha \int_0^T \norm{v^n(t)}_n^2 + \gamma\, l_{m+1}^n\big(v^n(t)\big) \dt\Bigg)}^2\\
&\quad\leq \EV{\exp \Big(6 \alpha \Big( l_2^n\big(v^n(0)\big) + l_2^n\big(w^n(0)\big) + T \big(\tr Q_1 + \tr Q_2 + C(L,m,\gamma)\big)\Big)\Big)},\\
\end{split}
\end{equation}
where we have replaced $\alpha$ by $6\alpha$ and the constant $C(L,m,\gamma)$ is a different one than in \eqref{eqproof1}. Finally, observe that for $v \in V^n$ the piecewise linear interpolation $\tv$ is in $V$ and its weak derivative is explicitly given by the differences $n(v_k - v_{k-1})$. Hence
\[
\norm{\tv}_V^2 = \int_0^1 \abs{\partial_\xi \tv (\xi)}^2 \dxi = \sum_{k=1}^n \int_{\frac{k-1}{n}}^{\frac{k}{n}} n^2 \big(v_k-v_{k-1}\big)^2 \dxi = n \sum_{k=1}^n \big(v_k-v_{k-1}\big)^2.
\]
Moreover, note that for $v_0, w_0 \in C([0,1])$ we can define the pointwise evaluation $v^n(0)$ and $w^n(0)$, which satisfy
\[
l_2^n\big( v^n(0) \big) \leq \norm{v_0}_{C([0,1])}^2 \quad \text{and}\quad l_2^n\big(w^n(0)\big) \leq \norm{w_0}^2_{C([0,1])}
\]
and we found an upper bound for \eqref{eqproof3} uniformly in $n \in \N$.
\end{proof}
Unlike $v$, the equation for $w$ has no regularizing linear part. Nevertheless, one can improve the a priori estimate significantly if the initial condition has more regularity, because the coupling with $v$ is Lipschitz continuous.
\begin{prop}\label{prop:wn}
For every $n \in \N$ and $1 \leq p < \infty$ the approximation $w^n$ satisfies the following improved a priori estimate 
\[
\EV{ \sup_{t \in [0,T]} \norm{w^n(t)}_n^p} \leq C_{\ref{prop:wn}}
\]
uniformly in $n \in \N$, where $C_{\ref{prop:wn}} = C(p,w_0,T,Q_2,L, C_{\ref{propAPriori}})$ is an explicitly known finite constant.
\end{prop}
\begin{proof}
Consider It\^o's formula for the difference $\abs{w_k^n(t) - w_{k-1}^n(t)}^2$. Directly plug in \eqref{A4} to obtain 
\begin{align*}
&\mathrm{d}|w_k^n(t) - w_{k-1}^n(t)|^2 \leq 2L \big(\tfrac{1}{n^2} + \abs{v_k^n(t) - v_{k-1}^n(t)}^2 + \abs{w_k^n(t) - w_{k-1}^n(t)}^2\big) \dt \\
&\qquad+ \frac{2}{\sqrt{n}} \sum_{l=1}^n \big( w_k^n(t) - w_{k-1}^n(t) \big) \Big(b_2 \big( \tfrac{k}{n}, v_k^n(t), w_k^n(t)\big) q^n_{2,k,l}\Big.\\
&\qquad\phantom{\frac{2}{\sqrt{n}} \sum_{l=1}^n \big( w_k^n(t) - w_{k-1}^n(t) \big)}\qquad \Big. - b_2 \big( \tfrac{k-1}{n}, v_{k-1}^n(t), w_{k-1}^n(t)\big) q^n_{2,k-1,l} \Big) \,\mathrm{d}\beta_{2,l}(s)\\
&\qquad + \frac{1}{n} \sum_{l=1}^n \Big(b_2 \big( \tfrac{k}{n}, v_k^n(t), w_k^n(t)\big) q^n_{2,k,l} - b_2 \big( \tfrac{k-1}{n}, v_{k-1}^n(t), w_{k-1}^n(t)\big) q^n_{2,k-1,l} \Big)^2 \dt.
\end{align*}
The It\^o correction term can be divided into two parts, each resembles a gradient in either $b_2$ or $q_2$. In particular, we have
\begin{align*}
&\frac{1}{n} \sum_{l=1}^n \Big(b_2 \big( \tfrac{k}{n}, v_k^n(t), w_k^n(t)\big) q^n_{2,k,l} - b_2 \big( \tfrac{k-1}{n}, v_{k-1}^n(t), w_{k-1}^n(t)\big) q^n_{2,k-1,l} \Big)^2\\
&\quad \leq \frac{2}{n} \sum_{l=1}^n \big(b_2 \big( \tfrac{k}{n}, v_k^n(t), w_k^n(t)\big) - b_2 \big( \tfrac{k-1}{n}, v_{k-1}^n(t), w_{k-1}^n(t)\big)\big)^2 \big(q^n_{2,k,l}\big)^2\\
&\qquad + \frac{2}{n} \sum_{l=1}^n \big(b_2 \big( \tfrac{k}{n}, v_k^n(t), w_k^n(t)\big)\big)^2 \big(q^n_{2,k,l} - q^n_{2,k-1,l}\big)^2\\
&\quad \leq 6 \norm{q_2}_{\infty}^2 \big( \tfrac{1}{n^2} + \abs{v_k^n(t) - v_{k-1}^n(t)}^2 + \abs{w_k^n(t) - w_{k-1}^n(t)}^2 \big)\\
&\qquad + \frac{2}{n} \sum_{l=1}^n \big(q^n_{2,k,l} - q^n_{2,k-1,l}\big)^2.
\end{align*}
In the last inequality we used that the kernel $q_2$ is $L^\infty$ and \eqref{A6}, i.\,e., $b_2$ Lipschitz and bounded. A summation over $k$ yields the inequality
\begin{equation}\label{eqproof4}
\begin{split}
\mathrm{d}\norm{w^n(t)}_n^2 &\leq \big(2L + 6 \norm{q_2}_{L^\infty}^2\big) \big( 1 + \norm{v^n(t)}_n^2 + \norm{w^n(t)}_n^2 \big) \dt\\
&\quad+ 2 \tr (-A)Q_2 \dt + \mathrm{d}M_t,
\end{split}
\end{equation}
where we denote the stochastic integral with $M_t$ for further reference. The trace appears naturally, since
\[
\sum_{k,l=1}^n \big(q^n_{2,k,l} - q^n_{2,k-1,l}\big)^2 \leq \int_0^1 \int_0^1 \abs{\partial_\xi q_2(\xi,\zeta)}^2 \dxi \dzeta \leq \norm{q_2}_{H^1((0,1)^2)}^2 = \tr (-A)Q_2
\]
by the fundamental theorem of calculus. We can use \eqref{eqproof4} to derive some estimate for the supremum in $t \in [0,T]$. Thus, consider the inequality to the power of $p > 1$,
\begin{align*}
\sup_{t \in [0,T]}\norm{w^n(t)}_n^{2p} &\leq 5^{p-1}\norm{w^n(0)}_n^{2p} + (5T)^{p-1} \big(2L + 6 \norm{q_2}_{L^\infty}^2 \big)^p  {\kern -0.2em} \int_0^T {\kern -0.4em}\sup_{s \in [0,t]} \norm{w^n(s)}_n^{2p} \dt\\
&\quad + 5^{p-1}\big(2L + 6 \norm{q_2}_{L^\infty}^2 \big)^p \Big(\int_0^T \big(1 + \norm{v^n(t)}_n^2\big) \dt\Big)^p\\
&\quad+ 5^{p-1}\big(2T \tr (-A)Q_2 \big)^p  + 5^{p-1}\sup_{t\in[0,T]} \abs{M_t}^p.
\end{align*}
Taking expectations on both sides and applying the Burkholder-Davis-Gundy inequality results in an estimate involving $\langle M\rangle_T$. Similar to the It\^o correction term above this is given by
\[
\langle M\rangle_T = 4 \int_0^T \norm{w^n(t)}_n^2 \Big( 6 \norm{q_2}_{L^\infty}^2 \big( 1 + \norm{v^n(t)}_n^2 + \norm{w^n(t)}_n^2 \big) + \tr (-A)Q_2 \Big) \dt,
\]
hence
\begin{align*}
\EV{\langle M \rangle_T^{\frac{p}{2}}} &\leq \frac{1}{2 C_p 5^{p-1}} \EV{\sup_{t\in[0,T]} \norm{w^n(t)}_n^2}+ 2 C_p 5^{p-1} T^{p-1} {\kern -0.2em}\int_0^T {\kern -0.3em}\EV{\sup_{s\in[0,t]}\norm{w^n(s)}_n^{2p}} \dt \\
&\quad + 2 C_p 5^{p-1} \EV{\Big( \int_0^T 1 + \norm{v^n(t)}_n^2 \dt \Big)^p},
\end{align*}
where $C_p$ denotes the constant in Burkholder-Davis-Gundy's inequality. In conclusion, we derived
\begin{align*}
\EV{\sup_{t\in[0,T]} \norm{w^n(t)}_n^{2p}} &\leq C_1 \EV{\norm{w_0}_{C^1([0,1])}^{2p}} + C_2 \int_0^T \EV{\sup_{s\in [0,t]} \norm{w^n(s)}_n^{2p}} \dt\\
&\quad + C_3 \EV{\Big( \int_0^T 1 + \norm{v^n(t)}_n^2 \dt \Big)^p}.
\end{align*}
Gronwall's lemma and Proposition \ref{propAPriori} yield the result.
\end{proof}
\section{Error Estimation, Proof of Theorem \ref{MainTheorem}}
This section is devoted to the proof of Theorem \ref{MainTheorem} and the estimation of the error in $\mathcal{L}^p (\Omega,\algF, \PP; C([0,T],\mathcal{H}))$. We proceed in several steps, first some pathwise estimates and then using the Burkholder-Davis-Gundy inequality together with the a priori information on the approximating solution $\tilde{X}^n$ in Propositions \ref{propAPriori} and \ref{prop:wn} to prove that $\tilde{X}^n$ converges strongly to the exact solution $X$.

\textbf{Step 1:} It\^o's formula applied to the square of the $\mathcal{H}$-norm (which is available for the variational solution) implies
\begin{align*}
&\norm{E^n(t)}_\mathcal{H}^2 - \norm{E^n(0)}_\mathcal{H}^2  = 2 \int_0^t \dualp{Av(s) - \iota_n A^n v^n(s)}{v(s) - \tv^n(s)} \ds\\
&\qquad+ 2 \int_0^t \phantom{\,}_{\mathcal{V}^\ast}\scp{\Phi\big(v(s),w(s)\big) - \iota_n\Phi^n \big(v^n(s),w^n(s)\big)}{E^n(s)}_\mathcal{V} \ds\\
&\qquad+ 2 \int_0^t \scp{E^n(s)}{\vek{B_1\big(v(s),w(s)\big)\sqrt{Q_1} - \iota_n b_1^n\big(v^n(s),w^n(s)\big)\sqrt{Q_1^n} P_n}{B_2\big(v(s),w(s)\big)\sqrt{Q_2} - \iota_n b_2^n\big(v^n(s),w^n(s)\big)\sqrt{Q_2^n} P_n} \dws}\\
&\qquad+ \sum_j \int_0^t \norm{B_j\big(v(s),w(s)\big) \sqrt{Q_j} - \iota_n b_j^n \big(v^n(s),w^n(s)\big) \sqrt{Q_j^n} P_n}_{L_2(H)}^2 \ds\\
&\quad= E_1 + E_2 + E_3 + E_4^1 + E_4^2.
\end{align*}
Each of the error terms $E_1, E_2, E_4^j$ can be split into two parts with the first one only involving the difference $E^n$ and the second one the approximation of the parameters of the equation. This corresponds to the monotonicity of the equation shown in Theorem \ref{thm:Existenz}. For the stochastic integral in $E_3$ we use the same argument after applying the Burkholder-Davis-Gundy inequality later. Now let us begin with, for instance $E_1$, where an integration by parts yields
\begin{align*}
E_1 &= -2 \int_0^t \norm{v(s) - \tv^n(s)}_V^2 \ds + 2 \int_0^t \dualp{A\tv^n(s) - \iota_n A^n v^n(s)}{v(s)-\tv^n(s)} \ds.
\intertext{Moreover, the one-sided Lipschitz condition \eqref{A1} for the nonlinear drift part and the Lipschitz continuity of $B$ in equation \eqref{LipschitzB} yield}
E_2 &\leq L {\kern -0.1em}\int_0^t{\kern -0.3em} \norm{E^n(s)}_\mathcal{H}^2 \ds + 2 {\kern -0.1em}\int_0^t {\kern -0.6em} \phantom{\,}_{\mathcal{V}^\ast}\scp{\Phi\big(\tv^n(s),\tw^n(s)\big) - \iota_n\Phi^n \big(v^n(s),w^n(s)\big)}{E^n(s)}_\mathcal{V} \ds\\
E_4^j &\leq 2 c_{\theta_j} \tr (-A)^{\theta_j} Q_j \int_0^t \norm{E^n(s)}_\mathcal{H}^2 \ds\\
&\quad + 2 \int_0^t \norm{B_j \big(\tv^n(s),\tw^n(s)\big) \sqrt{Q_j} - \iota_n b_j^n\big(v^n(s),w^n(s)\big)\sqrt{Q_j^n} P_n}_{L_2(H)}^2 \ds.
\end{align*}
In the following, we estimate each of the integrands in the error terms for fixed $s$, thus we drop the time dependence in the notation.

\textbf{Step 2: (Approximation error of the Laplacian)} At first, we will prove that the error term coming from the linear part converges to $0$ if $v \in V$, i.\,e., we need only the guaranteed regularity of the variational solution. The rate of convergence is not uniform and given implicitly in terms of the solution $v$. Second, if we assume additional regularity on $v$, we deduce an explicit rate in terms of $n$. Let $\varphi \in V$, then
\[
\dualp{A\tv^n - \iota_n A^n v^n}{\varphi} = \dualp{A\tv^n}{\varphi} - \scp{\iota_n A^n v^n}{\varphi}
\]
since the linear interpolation is in $V$. In the first term an integration by parts yields
\begin{equation}\label{eq:IBP}
\dualp{A\tv^n}{\varphi} = - n\sum_{k=1}^n \big(v_k^n - v_{k-1}^n\big) ( \varphi(\tfrac{k}{n}) - \varphi ( \tfrac{k-1}{n})).
\end{equation}
In the second term we use the summation by parts formula \eqref{SBP}, hence 
\begin{align}
\scp{\iota_n A^n v^n}{\varphi} &= \sum_{k=1}^n \int_{\frac{k-1}{n}}^{\frac{k}{n}} \big(A^n v^n \big)_k (n\xi-k+1) \varphi(\xi) + \big(A^n v^n \big)_{k-1} (k-n\xi) \varphi(\xi) \dxi \notag\\
&= \sum_{k=1}^n \big(A^n v^n \big)_k g_k + \big(A^n v^n \big)_{k-1} \bar{g}_{k-1}\notag\\
&= - n^2 \sum_{k=1}^n (v_k^n - v_{k-1}^n) \big( (g_k - g_{k-1}) + (\bar{g}_k - \bar{g}_{k-1})\big),\label{eq:SBP}
\end{align}
where we denote the integrals by $g_k$ and $\bar{g}_{k-1}$, respectively. Note that the boundary terms vanish because of the boundary conditions for $v^n$. 
The next step is to replace $\varphi$ by $E^n$ but since the computations differ for $v$ and the piecewise linear $\tv^n$ we split these terms and obtain the following lemmas.
\begin{lem}\label{lem:ErrorLaplaceVn}
For all $n \in \N$ it holds that
\[
\dualp{A\tv^n - \iota_n A^n v^n}{\tv^n} \leq 0.
\]
\end{lem}
\begin{proof}
Since $\tv^n$ is piecewise linear we can compute $g_k$ and $\bar{g}_k$, which are
\[
g_k = \frac{1}{3n} v_k^n + \frac{1}{6n} v_{k-1}^n \quad \text{and} \quad \bar{g}_k = \frac{1}{6n} v_{k+1}^n + \frac{1}{3n} v_k^n.
\]
Obviously, we also have $\tv^n \big(\tfrac{k}{n}\big) = v_k^n$ and $\tv^n\big(\tfrac{k-1}{n}\big) = v_{k-1}^n$. Equations \eqref{eq:IBP} and \eqref{eq:SBP} now read as
\begin{align*}
&\dualp{A\tv^n - \iota_n A^n v^n}{\tv^n} = - n\sum_{k=1}^n (v_k^n - v_{k-1}^n)^2 + \frac{2n}{3} \sum_{k=1}^n (v_k^n - v_{k-1}^n)^2\\
&\qquad+ \frac{n}{6} \sum_{k=2}^n (v_k^n - v_{k-1}^n)(v_{k-1}^n-v_{k-2}^n) + \frac{n}{6} \sum_{k=1}^{n-1} (v_k^n - v_{k-1}^n)(v_{k+1}^n-v_k^n).
\end{align*}
With Cauchy's inequality we can bound the latter terms by the former ones, thus the statement is proven.
\end{proof}
\begin{lem}\label{lem:ErrorLaplacePhi}
Let $\varphi \in V$, then
\[
\dualp{A\tv^n - \iota_n A^n v^n}{\varphi} \leq 2 \norm{\tv^n}_V \cdot I_n(\varphi),
\]
where
\[ 
I_n(\varphi) \df \Big( \sum_{k=1}^n \int_{\frac{k-1}{n}}^{\frac{k}{n}} \Big(\varphi'(\xi) - \varphi' \big(\xi \pm \tfrac{1}{n}\big) \Big)^2 \dxi \Big)^\frac12.
\]
\end{lem}
\begin{proof}
The factor with $\varphi$ in equations \eqref{eq:IBP} and \eqref{eq:SBP} can be written as
\begin{align*}
&\tfrac{1}{n} \Big( \varphi\big(\tfrac{k}{n}\big) - \varphi\big(\tfrac{k-1}{n}\big)\Big) - (g_k - g_{k-1}) + (\bar{g}_k - \bar{g}_{k-1})\notag\\
&\quad= \int_{\frac{k-1}{n}}^{\frac{k}{n}} (n\zeta - k+1) \Big[ \big( \varphi \big(\tfrac{k}{n}\big) - \varphi(\zeta)\big) - \big( \varphi\big(\tfrac{k-1}{n}\big) - \varphi\big(\zeta - \tfrac{1}{n}\big)\big) \Big]\notag\\
&\qquad+ (k-n\zeta) \Big[\big( \varphi (\zeta) - \varphi\big(\tfrac{k-1}{n}\big)\big) - \big( \varphi\big(\zeta + \tfrac{1}{n}\big) - \varphi\big(\tfrac{k}{n}\big)\big) \Big] \dzeta \fd c_n(\varphi),\notag
\end{align*}
because $(n\zeta - k+1) + (k-n\zeta) = 1$ and the integrals appearing in $g_{k-1}$ and $\bar{g}_k$ are shifted to the same interval. Rewrite $c_n(\varphi)$ with the fundamental theorem of calculus and use Jensen's inequality to obtain
\begin{align}
c_n(\varphi)^2 &= \Bigg(\int_{\frac{k-1}{n}}^{\frac{k}{n}} (n \zeta -k+1) \int_{\zeta}^{\frac{k}{n}} \varphi'(\xi) - \varphi'\big(\xi-\tfrac{1}{n}\big) \dxi \dzeta\Bigg. \notag\\
&\qquad\phantom{=}+\Bigg. \int_{\frac{k-1}{n}}^{\frac{k}{n}} (k-n\zeta) \int^{\zeta}_{\frac{k-1}{n}} \varphi'(\xi) - \varphi' \big(\xi+\tfrac{1}{n}\big) \dxi \dzeta\Bigg)^2 \notag\\
&\leq 2 n^{-3} \int_{\frac{k-1}{n}}^{\frac{k}{n}} \Big(\varphi'(\xi) - \varphi' \big(\xi \pm \tfrac{1}{n}\big) \Big)^2 \dxi.\label{eq:ErrorLaplacePhi1}
\end{align}
All there remains to do is an application of the Cauchy-Schwarz inequality.
\end{proof}
\begin{lem}
Let $I_n(\varphi)$ be defined as in Lemma \ref{lem:ErrorLaplacePhi}.
\begin{enumerate}
\item If $\varphi \in V$, then $I_n(\varphi) \leq 4 \norm{\varphi}_V$ and, moreover, $\lim_{n \to \infty} I_n(\varphi) = 0$.
\item If $\varphi \in H^{1+\alpha}(0,1)$, the explicit rate of convergence is given by $I_n(\varphi) \leq 2^{\alpha+2} n^{-\alpha} \norm{\varphi}_{H^{1+\alpha}}$.
\end{enumerate}
\end{lem}
\begin{proof}
At first, we will consider the second assertion. Thus, let $\varphi \in C^{1+\alpha}([0,1])$ and denote by $\varphi'$ its first derivative. Recall the definition of $c_n(\varphi)$ in the proof above. An additional integral in $\zeta$ over the same interval yields another $n$, while we can insert $\varphi'(\zeta)$ into the square above. The distance $\abs{\zeta \pm \tfrac1n - \xi}$ is of course bounded by $\tfrac2n$, hence
\[
c_n(\varphi)^2 \leq 2^{2\alpha + 3} n^{-3 - 2\alpha} \int_{\frac{k-1}{n}}^{\frac{k}{n}} \int_{\frac{k-1}{n}}^{\frac{k}{n}} \frac{\big(\varphi'(\zeta) - \varphi'(\xi)\big)^2}{\abs{\zeta-\xi}^{2\alpha+1}} + \frac{\big( \varphi'(\xi) - \varphi'\big(\zeta \pm \tfrac{1}{n}\big) \big)^2}{\abs{\zeta \pm \tfrac{1}{n} - \xi}^{2\alpha+1}} \dzeta \dxi
\]
and we have shown \textit{ii.}, since the sum of the latter expression is uniformly bounded by the Sobolev-Slobodeckij semi-norm of $\varphi$ in $H^{1+\alpha}$.

In a second step, we prove that if $\varphi \in V$, the expression $I_n(\varphi)$ is bounded uniformly in $n$ and we can indeed approximate $\varphi$ by sufficiently smooth functions in $V$ to obtain the desired convergence. For this purpose, we estimate \eqref{eq:ErrorLaplacePhi1} simply by
\[
c_n(\varphi)^2\leq 8 n^{-3} \int_{\frac{k-1}{n}}^{\frac{k}{n}} \abs{\varphi'(\xi)}^2 + \abs{\varphi'\big(\xi \pm \tfrac{1}{n}\big)}^2 \dxi,
\]
hence $I_n(\varphi) \leq 4 \norm{\varphi}_V$. Now, let $\eps>0$ be given. Clearly, $C^{1+\alpha}([0,1]) \subset V$ densely for $\alpha>0$, therefore one can find $\varphi_\eps \in C^{1+\alpha}([0,1])$ such that $\norm{\varphi - \varphi_\eps}_V \leq \frac{\eps}{8}$. Furthermore, we can find $n_0 \in \N$ such that for all $n \geq n_0$,
\[
I_n(\varphi) \leq I_n(\varphi-\varphi_\eps) + I_n(\varphi_\eps) \leq 4 \norm{\varphi-\varphi_\eps}_V + 2^{\alpha+2} n^{-\alpha} \norm{\varphi_\eps}_{H^{1+\alpha}}< \frac{\eps}{2} + \frac{\eps}{2} = \eps.\qedhere
\]
\end{proof}
If we turn our focus back on the original error term $E_1$ and apply the previous lemmas, we have shown that
\begin{equation}\label{EstimateE1}
E_1 \leq - 2 \int_0^t \norm{v(s) - \tv^n(s)}_V^2 \ds + 2 \int_0^t \norm{\tv^n(s)}_V I_n\big(v(s)\big) \ds.
\end{equation}

\textbf{Step 3: (Approximation error of the nonlinearities)} The nonlinear terms in $E_2$ are estimated with the (local) Lipschitz conditions \eqref{A3} and \eqref{A4} together with the uniform a priori estimates on the approximating solutions $\tv^n$ and $\tw^n$. The error term in $E_2$ is
\[
\phantom{\,}_{\mathcal{V}^\ast}\scp{\Phi\big(\tv^n(s),\tw^n(s)\big) - \iota_n\Phi^n \big(v^n(s),w^n(s)\big)}{E^n(s)}_\mathcal{V}
\]
which consists of contributing parts of $\phi_1$ and $\phi_2$. Since the assumptions on $\phi_2$ are much stronger, we exemplarily estimate the one for $\phi_1$ in the following:
\begin{align*}
&\sum_{k=1}^n \int_{\frac{k-1}{n}}^{\frac{k}{n}} \Big[ \phi_1 \big(\xi, \tv^n(\xi), \tw^n(\xi) \big) \Big.\\
&\qquad\Big.- (n\xi - k + 1) \phi_1 \big( \tfrac{k}{n}, v^n_k, w^n_k \big) - (k - n\xi) \phi_1 \big(\tfrac{k-1}{n}, v^n_{k-1}, w^n_{k-1}\big)\Big] \big(v(\xi) - \tv^n(\xi) \big) \dxi\\
&\quad\leq \sum_{k=1}^n \int_{\frac{k-1}{n}}^{\frac{k}{n}} \Big[ \abs{ (n\xi - k + 1)}\cdot \abs{\phi_1 \big(\xi, \tv^n(\xi), \tw^n(\xi) \big) - \phi_1 \big( \tfrac{k}{n}, v^n_k, w^n_k \big)} \Big.\\
&\qquad\Big.+ \abs{k - n\xi} \cdot \abs{\phi_1 \big(\xi, \tv^n(\xi), \tw^n(\xi) \big) - \phi_1 \big(\tfrac{k-1}{n}, v^n_{k-1}, w^n_{k-1} \big)}\Big] \abs{v(\xi) - \tv^n(\xi)} \dxi\\
&\quad\leq \sum_{k=1}^n \int_{\frac{k-1}{n}}^{\frac{k}{n}} \tfrac{L}{2} \Big[ \big( 1 + \abs{v_k^n}^{m-1} + \abs{v_{k-1}^n}^{m-1}\big) \big( \tfrac1n + \abs{v_k^n-v_{k-1}^n}\big)\Big.\\
&\qquad\Big. + \abs{w_k^n - w_{k-1}^n} \Big] \abs{v(\xi) - \tv^n(\xi)} \dxi, 
\end{align*}
where we used \eqref{A3}, $\abs{\tv^n(\xi)} \leq \abs{v_k^n} + \abs{v_{k-1}^n}$ and that $\abs{\tv^n(\xi) - v_k^n} = (k-n\xi) \abs{v_k^n - v_{k-1}^n}$ as well as $\abs{\tv^n(\xi) - v_{k-1}^n} = (n\xi-k+1) \abs{v_k^n - v_{k-1}^n}$. Of course, the same holds for $\tw^n$. With Young's inequality and $\alpha^\ast$ from Proposition \ref{propAPriori} we can further estimate
\begin{align*}
&\leq \tfrac{\alpha^\ast}{p} n \sum_{k=1}^n \big(\tfrac{1}{n^2} + \abs{v_k^n - v_{k-1}^n}^2\big) \int_{\frac{k-1}{n}}^{\frac{k}{n}} \abs{v(\xi) - \tv^n(\xi)}^2 \dxi\\
&\quad + \tfrac{L^2 p}{8 \alpha^\ast n^2} \sum_{k=1}^n \big(1 + \abs{v_k^n}^{m-1}+ \abs{v_{k-1}^n}^{m-1}\big)^2 + \tfrac{L}{4n^2} \norm{\tw^n}_V^2 + \tfrac{L}{4} \norm{v-\tv^n}_H^2\\
&\leq \Big( \tfrac{\alpha^\ast}{p} \big( 1 + \norm{\tv^n}_V^2 \big) + \tfrac{L}{4}\Big) \norm{v-\tv^n}_H^2 + \tfrac{3L^2 p}{4 \alpha^\ast n} \Big(1 + l_{m+1}^n\big(v^n\big)\Big) + \tfrac{L}{4n^2} \norm{\tw^n}_V^2,
\end{align*}
since $2(m-1) \leq m+1$ for $m\leq 3$. We need to remark at this point, that the superlinear growth of $\phi_1$ in $v$ (of order $m$ in \eqref{A3}) is the reason for the importance of the exponential a priori estimate in Proposition \ref{propAPriori}. Compared to the Lipschitz part in $w$ there appears a nonconstant coefficient in front of the error and the Gronwall type argument relies on this exponential integrability.

$\phi_2$'s part of the error can be obtained in the same way as for $w$ above, hence
\begin{equation}\label{EstimateNonlinear}
\begin{split}
E_2 &\leq \int_0^t \Big( \tfrac{\alpha^\ast}{p} \big( 1 + \norm{\tv^n(s)}_V^2\big) + L\Big) \norm{E^n(s)}_\mathcal{H}^2 \ds\\
&\quad+ \tfrac{3 L^2 p}{4 \alpha^\ast n} \int_0^t 1 + l_{m+1}^n \big( v^n(s) \big) \ds + \tfrac{L}{n^2} \int_0^t 1 + \norm{\tv^n(s)}_V^2 + \norm{\tw^n(s)}_V^2 \ds.
\end{split}
\end{equation}

\textbf{Step 4: (Approximation of the covariance operator)} The It\^o correction terms contain a Hilbert-Schmidt norm which, in our case, is given by the $L^2$-norm of the associated kernels. Thus, we can write these parts of the error terms $E_4^j$ similarly to the one of the nonlinear part in Step 3. Recall that $P_n h = (\scp{h}{\sqrt{n}1_{[\frac{l-1}{n}, \frac{l}{n})}})_{1 \leq l \leq n}$.
\begin{align*}
&\norm{ B_j(\tv^n,\tw^n) \sqrt{Q_j} - \iota_n b_j^n(v^n,w^n)\sqrt{Q_j^n} P_n}^2_{L_2(H)}\\
&= \iint_0^1 \Bigg( b_j\big(\xi, \tv^n(\xi), \tw^n(\xi) \big) q_j(\xi,\zeta) - \sum_{k,l=1}^n 1_{[\frac{k-1}{n},\frac{k}{n})}(\xi)1_{[\frac{l-1}{n}, \frac{l}{n})}(\zeta)\Bigg.\\
& \Bigg. \cdot \Big[ (n\xi - k+1) b_j\big( \tfrac{k}{n}, v_k^n,w_k^n \big) q^n_{j,k,l} + (k-n\xi) b_j\big( \tfrac{k-1}{n}, v_{k-1}^n, w_{k-1}^n\big) q^n_{j,k-1,l} \Big]\Bigg)^2{\kern -0.5em} \dzeta \dxi.
\end{align*}
A closer look reveals that the error consists of parts $S_1^j$ contributed by the approximation of $b_j$ and $S_2^j$ by the approximation of the kernel $q_j$. Therefore, we split the norm square into two parts where the first one is
\begin{align*}
S_1^j &= 2 \sum_{k=1}^n \int_0^1 \int_{\frac{k-1}{n}}^{\frac{k}{n}} \Big( b_j\big(\xi, \tv^n(\xi), \tw^n(\xi)\big) - (n\xi - k+1) b_j\big( \tfrac{k}{n}, v_k^n, w_k^n \big) \Big.\\
&\quad \Big.- (k-n\xi) b_j\big( \tfrac{k-1}{n}, v_{k-1}^n, w^n_{k-1}\big) \Big)^2 q_j(\xi,\zeta)^2 \dzeta \dxi.
\end{align*}
Similarly to Step 3 for the nonlinear drift term, since $b$ is Lipschitz continuous,
\begin{align*}
&\Big( b_j\big(\xi, \tv^n(\xi), \tw^n(\xi)\big) - (n\xi - k+1) b_j\big( \tfrac{k}{n}, v_k^n, w_k^n \big) - (k-n\xi) b_j\big( \tfrac{k-1}{n}, v_{k-1}^n, w_{k-1}^n \big) \Big)^2\\
&\quad\leq \tfrac12 \big( \tfrac{1}{n^2} + \abs{v_k^n - v_{k-1}^n}^2 + \abs{w_k^n - w_{k-1}^n}^2\big)
\end{align*}
and the summation over $k$ yields
\[
S^j_1 \leq \tfrac{1}{n} \tr Q_j \big( 1 + \norm{\tv^n}_V^2 + \norm{\tw^n}_V^2 \big).
\]
The second error $S^j_2$ term is due to the approximation of the covariance kernel and given by
\begin{align*}
S_2^j &= 4\sum_{k,l=1}^n \int_{\frac{k-1}{n}}^{\frac{k}{n}} \int_{\frac{l-1}{n}}^{\frac{l}{n}} \Big( b_j\big( \tfrac{k}{n}, v_k^n, w_k^n\big) \big( q_j(\xi,\zeta) - q^n_{j,k,l}\big)\Big)^2\\
&\quad + \Big(b_j\big( \tfrac{k-1}{n}, v_{k-1}^n,w_{k-1}^n\big) \big( q_j(\xi,\zeta) - q^n_{j,k-1,l}\big)\Big)^2 \dzeta \dxi,
\end{align*}
where we already used $(n\xi - k+1) \leq 1$ as well as $(k-n\xi) \leq 1$. Exemplarily we do the estimate for the part with $q^n_{j,k,l}$ but note that the procedure remains the same in the one with $q^n_{j,k-1,l}$ except for a constant. Since $b_j$ is bounded, we obtain
\begin{align*}
S_2^j &\leq 4 \sum_{k,l=1}^n \int_{\frac{k-1}{n}}^{\frac{k}{n}} \int_{\frac{l-1}{n}}^{\frac{l}{n}} \big( q_j(\xi,\zeta) - q^n_{j,k,l} \big)^2 \dzeta \dxi\\
&= 4 \sum_{k,l=1}^n \int_{\frac{k-1}{n}}^{\frac{k}{n}} \int_{\frac{l-1}{n}}^{\frac{l}{n}} \Big( q_j(\xi,\zeta) - n^2 \int_{\frac{k-1}{n}}^{\frac{k}{n}} \int_{\frac{l-1}{n}}^{\frac{l}{n}} q_j(\rho,\eta) \,\mathrm{d}\rho \,\mathrm{d}\eta \Big)^2 \dzeta \dxi\\
&\leq 4 n^2 \sum_{k,l=1}^n \int_{\frac{k-1}{n}}^{\frac{k}{n}} \int_{\frac{l-1}{n}}^{\frac{l}{n}} \int_{\frac{k-1}{n}}^{\frac{k}{n}} \int_{\frac{l-1}{n}}^{\frac{l}{n}} \abs{q_j(\xi,\zeta) - q_j(\rho,\eta)}^2 \,\mathrm{d}\rho \, \mathrm{d}\eta \dzeta \dxi.
\intertext{Since $\abs{\xi-\rho} \leq \tfrac{1}{n}$ as well as $\abs{\zeta - \eta} \leq \tfrac{1}{n}$ we can expand this expression to the Sobolev-Slobodeckij semi-norm of $q_j$.}
&\leq \tfrac{16}{n} \iiiint_0^1 \frac{\abs{q_j(\xi,\zeta) - q_j(\rho,\eta)}^2}{\big( \abs{\xi-\rho}^2 + \abs{\zeta-\eta}^2\big)^{\frac32}} \,\mathrm{d}\rho \, \mathrm{d}\eta \dzeta \dxi \\
&\leq \tfrac{16}{n} \norm{q_j}_{H^{\frac12}((0,1)^2)}^2 \leq \tfrac{16}{n} \tr (-A)^{\theta_j}Q_j.
\end{align*}
In summary, we have shown that (the modified constants are due to the omitted terms)
\begin{equation}\label{EstimateCovariance}
\begin{split}
E_4^j &\leq 2c_{\theta_j} \tr (-A)^{\theta_j} Q_j \int_0^t \norm{E^n(s)}_\mathcal{H}^2 \ds + \tfrac{64 t}{n} \tr (-A)^{\theta_j} Q_j\\
&\quad+ \tfrac{2}{n} \tr Q_j \int_0^t \big(1 + \norm{\tv^n(s)}_V^2 + \norm{\tw^n(s)}_V^2 \big) \ds 
\end{split}
\end{equation}

\textbf{Step 5:} Control over the supremum

The four previous steps allow a first estimate nonuniform in $t$. For this purpose define the processes
\begin{align*}
L^n_p(t) &\df \alpha^\ast \big(1 + \norm{\tv^n(t)}_V^2\big) + 2p^2 \sum_j c_{\theta_j} \tr (-A)^{\theta_j} Q_j\\
\intertext{and}
K^n_p(t) &\df 2 \norm{\tv^n(t)}_V I_n\big(v(t)\big) + \frac{3 L^2 p^2}{4 \alpha^\ast n} \big(1 + l_{m+1}^n\big(v^n(t)\big) \big) \\
& + \frac{p^2}{n} \Big(L + 2 \sum_j \tr Q_j\Big) \big(1 + \norm{\tv^n(t)}_V^2 + \norm{\tw^n(t)}_V^2 \big) + \frac{64 p^2}{n} \sum_j \tr (-A)^{\theta_j} Q_j.
\end{align*}
Now we apply It\^o's formula with the function $f(x) = x^p$, $p\geq 1$ for the square of the $\mathcal{H}$-norm of the error $E^n(t)$.
\begin{equation}\label{ItoLp}
\mathrm{d}\norm{E^n(t)}_{\mathcal{H}}^{2p} = p \norm{E^n(t)}_{\mathcal{H}}^{2p-2} \, \mathrm{d} \norm{E^n(t)}_{\mathcal{H}}^2 + 2p(p-1) \norm{E^n(t)}_{\mathcal{H}}^{2p-4} \, \mathrm{d}\langle M\rangle_t,
\end{equation}
where the local martingale $M(t)$ is given by $E_3$, i.\,e.,
\begin{align}
\mathrm{d}M(t) &\df \scp{E^n(t)}{\vek{B_1\big(v(t),w(t)\big)\sqrt{Q_1} - \iota_n b_1^n\big(v^n(t),w^n(t)\big)\sqrt{Q_1^n} P_n}{B_2\big(v(t),w(t)\big)\sqrt{Q_2} - \iota_n b_2^n\big(v^n(t),w^n(t)\big)\sqrt{Q_2^n} P_n} \dwt},\notag\\
\intertext{and  quadratic variation bounded from above by}
\mathrm{d}\langle M \rangle_t &\leq \norm{E^n(t)}_{\mathcal{H}}^2 \sum_j \norm{B_j\big(v(t),w(t)\big) \sqrt{Q_j} - \iota_n b_j^n \big(v^n(t),w^n(t)\big) \sqrt{Q_j^n} P_n}_{L_2(H)}^2 \dt\notag\\
\begin{split}
&\leq 2 \sum_j c_{\theta_j} \tr (-A)^{\theta_j} Q_j \norm{E^n(t)}_\mathcal{H}^4 \dt + \tfrac{64}{n} \sum_j \tr (-A)^{\theta_j} Q_j \norm{E^n(t)}_\mathcal{H}^2 \dt\\
&\quad+ \tfrac{2}{n} \sum_j \tr Q_j \big(1 + \norm{\tv^n(t)}_V^2 + \norm{\tw^n(t)}_V^2 \big) \norm{E^n(t)}_\mathcal{H}^2 \dt,
\end{split}\label{estQuadVar}
\end{align}
which were essentially the estimates of $E_4^j$. Thus \eqref{ItoLp}, It\^o's product rule, \eqref{estQuadVar} and Steps 1--4 imply
\begin{align*}
&\sup_{t \in [0,T]} e^{-\int_0^t L^n_p(s)\ds} \norm{E^n(t)}_\mathcal{H}^{2p} \leq \norm{E^n(0)}_\mathcal{H}^{2p}+ \int_0^T e^{-\int_0^t L^n_p(s) \ds} K^n_p(t) \norm{E^n(t)}_{\mathcal{H}}^{2p-2} \dt\\
&\qquad + 2p \sup_{t \in [0,T]} \left|\int_0^t e^{-\int_0^s L^n_p(r)\dr} \norm{E^n(s)}_{\mathcal{H}}^{2p-2} \,\mathrm{d}M(s)\right|\\
&\quad \leq \norm{E^n(0)}_{\mathcal{H}}^2 + \tfrac{p-1}{p} \sup_{t \in [0,T]} e^{-\int_0^t L^n_p(s) \ds} \norm{E^n(t)}_{\mathcal{H}}^{2p} + \tfrac1p \Bigg( \int_0^T K^n_p(t) \dt \Bigg)^p\\
&\qquad + 2p \sup_{t \in [0,T]} \left|\int_0^t e^{-\int_0^s L^n_p(r)\dr} \norm{E^n(s)}_{\mathcal{H}}^{2p-2} \,\mathrm{d}M(s)\right|.
\end{align*}
The second inequality is due to Young's inequality and we can absorb the second summand by the left-hand side. Now, take the expectation and Burkholder-Davis-Gundy's inequality bounds the supremum of the stochastic integral from above by its quadratic variation; more precisely,
\begin{equation}\label{estproof1}
\begin{split}
&\EV{\sup_{t \in [0,T]} e^{-\int_0^t L^n_p(s)\ds} \norm{E^n(t)}_{\mathcal{H}}^{2p}} \leq  p\, \EV{\norm{E^n(0)}_{\mathcal{H}}^{2p}}\\
& + \EV{ \Big(\int_0^T K^n_p(t)\dt\Big)^p} + 4p\,\EV{ \Big( \int_0^T e^{-2 \int_0^t L^n_p(s)\ds} \norm{E^n(t)}_{\mathcal{H}}^{4p-4} \,\mathrm{d}\langle M\rangle_t\Big)^\frac12}.
\end{split}
\end{equation}
With the bound on the quadratic variation from \eqref{estQuadVar} and Young's inequality we can estimate the latter summand in terms of the left-hand side and $K^n_p$ as follows:
\begin{align*}
&4p\,\EV{ \Bigg( \int_0^T e^{-2 \int_0^t L^n_p(s)\ds} \norm{E^n(t)}_{\mathcal{H}}^{4p-4} \,\mathrm{d}\langle M\rangle_t\Bigg)^\frac12}\\
&\quad \leq \frac12 \EV{\sup_{t \in [0,T]} e^{- \int_0^t L_p^n(s) \ds} \norm{E^n(t)}_\mathcal{H}^{2p}}\\
&\qquad + \big(2p \sum_j c_{\theta_j} \tr (-A)^{\theta_j} Q_j\big)  \int_0^T \EV{\sup_{s \in [0,t]} e^{- \int_0^s L_p^n(r) \dr} \norm{E^n(s)}_\mathcal{H}^{2p}} \dt\\
&\qquad + 2 (16p-8)^{2p-1} \EV{ \Big( \int_0^T K^n_p(t) \dt \Big)^p}.
\end{align*}
Thus, we can apply Gronwall's inequality for any $p \geq 1$ to obtain
\begin{equation}\label{EstimateSup}
\begin{split}
&\EV{\sup_{t \in [0,T]} e^{-\int_0^t L^n_p(s)\ds} \norm{E^n(t)}_{\mathcal{H}}^{2p}}\\
&\quad\leq e^{C_1 T} \Bigg( 2p \EV{\norm{E^n(0)}_{\mathcal{H}}^{2p}} + C_2 \EV{\Big(\int_0^T K^n(t) \dt\Big)^p}\Bigg),
\end{split}
\end{equation}
with $C_1 = 4pT \sum_j c_{\theta_j} \tr (-A)^{\theta_j} Q_j$ and $C_2 = 4(16-8)^{2p-1} + 2$. This preliminary error estimate yields the desired one via H\"older's inequality provided the right-hand side of \eqref{EstimateSup} is finite since we can control the exponential by Proposition \ref{propAPriori}. Therefore, we have for $p\geq 1$,
\begin{align}
\EV{\sup_{t \in [0,T]} \norm{E^n(t)}_{\mathcal{H}}^p}^\frac1p &\leq \EV{\sup_{t \in [0,T]} e^{-\frac12 \int_0^t L^n_p(s)\ds} \norm{E^n(t)}_{\mathcal{H}}^p e^{\frac12 \int_0^T L^n_p(t)\dt}}^\frac1p\notag\\
&\leq \EV{\sup_{t \in [0,T]} e^{-\int_0^t L^n_p(s)\ds} \norm{E^n(t)}_{\mathcal{H}}^{2p}}^{\frac{1}{2p}} \EV{e^{\int_0^T L^n_p(t)\dt}}^{\frac{1}{2p}}\notag\\
&\leq C_{\ref{propAPriori}}^{\frac{1}{2p}}\, e^{C_1T}\Bigg( 2p\EV{\norm{E^n(0)}_{\mathcal{H}}^{2p}} + C_2\EV{\Big(\int_0^T K_p^n(t) \dt\Big)^p}\Bigg)^{\frac{1}{2p}}\label{finalEstimate}
\end{align}
and it remains to study the convergence of $K^n$ to $0$. For this purpose, we fix $p=1$ at first. Then, it follows by Propositions \ref{propAPriori} and \ref{prop:wn} that
\[
\EV{\int_0^T K_1^n(t) \dt} \leq \frac{C}{n} + 2 \EV{\int_0^T \norm{\tv^n(t)}_V I_n\big(v(t)\big) \dt}
\]
Since by Lemma \ref{lem:ErrorLaplacePhi} $I_n(v(t)) \leq 4 \norm{v(t)}_V$ and both $v$ and $\tv^n$ are (uniformly) bounded in $\mathcal{L}^2 (\Omega,\algF,\PP; L^2([0,T],V))$, we can use Lebesgue's dominated convergence theorem to deduce
\[
\EV{\int_0^T K^n(t) \dt} \to 0, \quad \text{hence}\quad \EV{\sup_{t \in [0,T]} \norm{E^n(t)}_{\mathcal{H}}} \to 0 \quad \text{as } n \to \infty
\]
by Lemma \ref{lem:ErrorLaplacePhi}. Thus, the exponential moment estimates carry over to the limit $v$ and we can apply the dominated convergence theorem in cases $p>1$ which concludes the first assertion of Theorem \ref{MainTheorem}. If $v \in \mathcal{L}^{p^\ast} (\Omega,\algF,\PP; L^2([0,T],H^{1+\alpha}))$ for some $\alpha > 0$ and $p^\ast > 1$, then again by Lemma \ref{lem:ErrorLaplacePhi} and Proposition \ref{propAPriori} we get
\begin{align*}
&\EV{\Big(\int_0^T K^n(t) \dt\Big)^p} \leq \frac{C}{n^p} + C' \EV{\Big(\int_0^T \norm{\tv^n(t)}_V^2\dt \Big)^{\frac{p}{2}} \Big( \int_0^T I_n\big(v(t)\big)^2 \dt\Big)^{\frac{p}{2}}}\\
&\quad\leq \frac{C}{n^p} + C'\, \EV{\Big(\int_0^T \norm{\tv^n(t)}_V^2\dt \Big)^{\frac{p p^\ast}{2(p^\ast-p)}}}^{\frac{p^\ast-p}{p^\ast}} {\kern -0.9em} \EV{\Big( \int_0^T I_n\big(v(t)\big)^2 \dt\Big)^{\frac{p^\ast}{2}}}^{\frac{p}{p^\ast}}\\
&\quad\leq \frac{C}{n^p} + \frac{C''}{n^{\alpha p}} \,\EV{\Big( \int_0^T \norm{v(t)}_{H^{1+\alpha}}^2 \dt \Big)^{\frac{p^\ast}{2}}}^{\frac{p}{p^\ast}},
\end{align*}
hence the second assertion is proven for all $p < p^\ast$.\qed

\section{Applications}

As an application for our results in Theorem \ref{MainTheorem}, we consider the spatially extended FitzHugh-Nagumo system with noise. Originally, this was stated by FitzHugh \cite{FitzHugh1} as a system of ODEs simplifying the famous Hodgkin-Huxley model \cite{HodgkinHuxley} for the generation of an action potential in a neuron in terms of a voltage variable $v$ and a so-called recovery variable $w$. Its spatially extended version is a model for the propagation of the action potential in the axon of a neuron. See, e.\,g., the monographs \cite{Ermentrout, Keener} for more details on the deterministic case. Now consider this system subject to external noise only in the voltage variable $v$. Together with the original parameters from \cite{FitzHugh2} this reads as
\begin{equation}\label{eq:FHN}
\begin{split}
\dot{v}(t,\xi) &= \partial_\xi^2 v(t,\xi) + v(t,\xi) - \frac13 v(t,\xi)^3 - w(t,\xi) + \eta(t,\xi),\\
\dot{w}(t,\xi) &= 0.08 \big( v(t,\xi) - 0.8 w(t,\xi) + 0.7 \big), \quad t \geq 0, \xi \in (0,L)
\end{split}
\end{equation}
equipped with homogeneous Neumann boundary conditions in $0$ and $L$. The noise $\eta$ is modeled by $\sqrt{Q} W(t)$ with a cylindrical Wiener process $W$ on $H$ and $Q$ to be specified below. One can immediately see that \eqref{eq:FHN} is of the type \eqref{eq:SRDE}. The first mathematical rigorous analysis of this equation in the context of mild solutions can be found in \cite{BonaccorsiFHN}. It has been observed, e.\,g., in \cite{TuckwellFHN} that this system has traveling pulse solutions (Figure \ref{fig}) which may break down due to the influence of the external noise, hence there is no transmission of the signal from $0$ to $L$. This phenomenon is usually referred to as the propagation failure and one is interested in calculating its probability depending on the strength of the external noise; see \cite{TuckwellFHN} for a heuristic approach.

\begin{figure}[ht]
\centering
\includegraphics[scale=0.5]{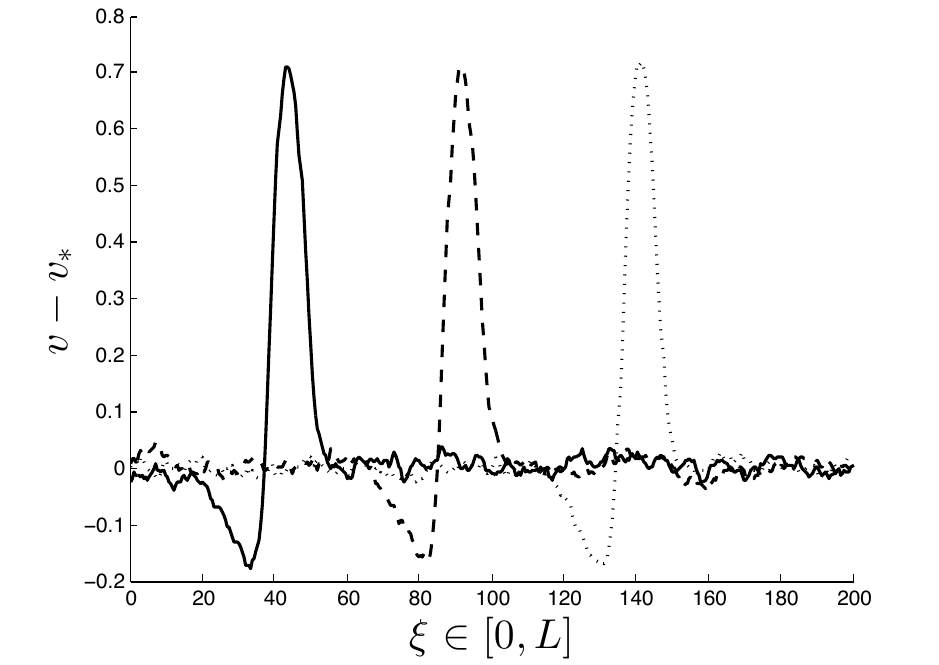}
\caption[Test]{A traveling pulse solution of \eqref{eq:FHN} obtained by the scheme \eqref{eq:NumSchemeFHN} propagating along $(0,L)$ at three different times (solid, dashed, dotted).}
\label{fig}
\end{figure}

\textbf{Numerical Approximation:} The numerical approximation in the study \cite{TuckwellFHN} is done via finite difference approximations in space and the Euler-Maruyama scheme in time. Set $\Delta x = L/n$ and $\Delta t = T/(4n^2)$ and denote by $(x_i)_i$,$(t_j)_j$ the equidistant grids corresponding to this. Approximating $v$ and $w$ in $(x_i, t_j)$ by $v_{i,j}$ and $w_{i,j}$ results in the scheme
\begin{equation}\label{eq:NumSchemeFHN}
\begin{split}
v_{i,j} &= v_{i,j-1} + \tfrac14\big( v_{i+1,j-1} - 2 v_{i,j-1} + v_{i-1,j-1} \big)\\
&\quad+ \phi_1 \big(x_i, v_{i,j-1}, w_{i,j-1}\big) \Delta t + 2\sigma \sqrt{n} N_{i,j},\\
w_{i,j} &= w_{i,j-1} + \phi_2 \big(x_i, v_{i,j-1}, w_{i,j-1}\big) \Delta t
\end{split} 
\end{equation}
with iid standard normal random variables $N_{i,j}$. Note that the author assumes that $\eta$ is space-time white noise, hence the approximated noise only has this simple structure.

\textbf{The Covariance Operator:} Although the noise is supposed to be white in time and space, we may use Theorem \ref{MainTheorem} to deduce convergence of the spatial approximation. Using properly scaled versions of the bump function
\[
\Psi(\xi, \zeta) =\begin{cases} \exp \Big(- \big(1-\xi^2-\zeta^2\big)^{-1}\Big), & \text{for }\xi^2 + \zeta^2 < 1,\\ 0, &\text{else,}\end{cases}
\]
one can construct a smooth kernel $q$ seeing only local interactions up to some distance $n^\ast$ and satisfying $q^n_{k,l} = n \delta_{k,l}$ for $n \leq n^\ast$, thus the numerical approximations using $Q$ and $I$, i.\,e., space-time white noise, as covariance operator do not differ.

\textbf{An Estimator for the Propagation Failure:} An appropriate estimator detecting the event of a propagation failure is given by the integral
\[
\Phi (v) \df \int_0^L v(\xi) - v_\ast \dxi,
\]
which significantly differs in cases with or without the traveling pulse based on the observation from Figure \ref{fig}. Here, $v_\ast \approx -1.1994$ is the voltage component of the unique real equilibrium point for the drift in \eqref{eq:FHN}. Given $\Phi$, the event of a propagation failure can be defined by 
\[
\Phi\big( v(t)\big) \leq \kappa \quad \text{for some } T_0 \leq t \leq T,
\]
for some appropriate threshold $\kappa > 0$ and some initializing time $T_0$. The quantity of interest is the probability
\[
\Prob{\min_{T_0 \leq t \leq T} \Phi\big(v(t)\big) \leq \kappa} \fd p_{Q,\kappa}
\]
of propagation failure depending on the noise covariance $Q$. Essentially we have to estimate the parameter of a Bernoulli distributed random variable. The sample average
\[
\hat{p}_{Q,\kappa}(\hat{v}) \df \frac1m \sum_{k=1}^m 1_{\{\min_{T_0 \leq t \leq T} \Phi(v^{(k)}(t)) \leq \kappa\}}
\]
based on $m$ iid copies $\hat{v} \df (v^{(k)})_{1 \leq k \leq m}$ is a natural estimator. Since the solution $v$ is not given explicitly, we can approximate $\hat{p}_{Q,\kappa}$ based on independent numerical observations $\hat{v}^n \df (\tv^{n,(k)})_{1 \leq k \leq m}$, as independent realizations of $\tv^n$ only. Besides the statistical Monte Carlo error, there appears additional uncertainty due to the approximation of the exact solution. Theorem \ref{MainTheorem} now allows us to quantify this.
\begin{cor}
Let $\hat{p}^n_{Q,\kappa} \df \hat{p}_{Q,\kappa}(\hat{v}^n)$ and $\eps >0$ be a priori given. Furthermore, assume that the solution $v \in \mathcal{L}^p( \Omega, \algF, \PP; L^2([0,T],H^2(0,1)))$ for some $p>2$. Then, a confidence interval for the estimation $\hat{p}^n_{Q,\kappa-\eps}$ of $p_{Q,\kappa}$ with confidence level $\alpha$ is given by $[ \hat{p}^n_{Q,\kappa-\eps} - \gamma, \hat{p}^n_{Q,\kappa-\eps} + \gamma]$ where
\[
\gamma = (\alpha m)^{-\frac12} \Big( 1 + 4\eps^{-2} C_{\ref{MainTheorem}} n^{-1}\Big)^\frac12.
\]
\end{cor}
\begin{rem}
The additional regularity of the solution may be obtained in the context of mild solutions as in \cite{BonaccorsiFHN} if $Q$ is sufficiently regular, since the heat semigroup in the equation for $v$ is analytic and therefore maps $H$ to $D(A) = H^2(0,1)$. However, this needs further investigation.
\end{rem}
\begin{proof}
We can estimate the probability of $p_{Q,\kappa}$ being outside of an interval of size $\gamma$ around the estimator $\hat{p}^n_{Q,\kappa-\eps}$ by Chebychev's inequality with
\begin{align*}
&\Prob{\abs{p_{Q,\kappa} - \hat{p}^n_{Q,\kappa-\eps}}> \gamma} = \Prob{\abs{p_{Q,\kappa} - \hat{p}_{Q,\kappa} + \hat{p}_{Q,\kappa} - \hat{p}^n_{Q,\kappa-\eps}}> \gamma}\\
&\quad\leq \Prob{\abs{p_{Q,\kappa} - \hat{p}_{Q,\kappa}} > \tfrac{\gamma}{2}} + \Prob{\abs{\hat{p}_{Q,\kappa} - \hat{p}^n_{Q,\kappa-\eps}}> \tfrac{\gamma}{2}}\\
&\quad\leq \frac{1}{\gamma^2 m} + \Prob{ \sum_{k=1}^m 1_{\{ \sup_{t \in [0,T]} \abs{\Phi (v^{(k)}(t)) - \Phi( \tv^{n,(k)}(t))} > \eps\}} > m\tfrac{\gamma}{2}}\\
&\quad\leq \frac{1}{\gamma^2 m} + \frac{4}{\gamma^2 \eps^2 m} \EV{\sup_{t \in [0,T]} \Big(\Phi (v^{(k)}(t)) - \Phi( \tv^{n,(k)}(t))\Big)^2}.
\end{align*}
The additional uncertainty due to the approximation of the exact solution is the mean squared error of $\Phi \big( \tv^n(t)\big)$.
\begin{align*}
&\EV{ \sup_{t \in [0,T]} \Big( \Phi\big(v(t)\big) - \Phi\big(\tv^n(t)\big)\Big)^2} = \EV{\sup_{t \in [0,T]} \Big( \int_0^L v(t,\xi) - \tv^n(t,\xi) \dxi \Big)^2}\\
&\quad\leq L\,\EV{\sup_{t \in [0,T]} \norm{v(t) - \tv^n(t)}_H^2} \leq L\, \EV{\sup_{t \in [0,T]} \norm{E^n(t)}_{\mathcal{H}}^2}.
\end{align*}
Theorem \ref{MainTheorem} now implies the convergence rate of $n^{-1}$ and we set
\[
\alpha = \frac{1}{\gamma^2 m} + \frac{4L}{\gamma^2 \eps^2 m} C_{\ref{MainTheorem}} n^{-1}.\qedhere
\]
\end{proof}

\section*{Acknowledgment}
This work was supported by the BMBF, FKZ 01GQ1001B. Furthermore, the authors thank the two referees for their valuable suggestions that helped to improve the article.


\begin{thebibliography}{10}
\setlength{\itemsep}{-4pt}
\footnotesize

\bibitem{BonaccorsiFHN}
S.~Bonaccorsi and E.~Mastrogiacomo.
\newblock {Analysis of the Stochastic FitzHugh-Nagumo System}.
\newblock {\em Infin. Dimens. Anal. Quantum Probab. Relat. Top.},
  11(3):427--446, 2008.

\bibitem{CarelliProhl}
E.~Carelli and A.~Prohl.
\newblock {Rates of Convergence for Discretizations of the Stochastic Incompressible Navier-Stokes Equations}.
\newblock {\em SIAM J. Numer. Anal.}, 50(5):2467--2496, 2012.

\bibitem{Ermentrout}
G.~B. Ermentrout and D.~H. Terman.
\newblock {\em {Mathematical Foundations of Neuroscience}}, volume~35.
\newblock Springer, 2010.

\bibitem{FitzHugh1}
R.~FitzHugh.
\newblock {Impulses and Physiological States in Theoretical Models of Nerve
  Membrane}.
\newblock {\em Biophys. J.}, 1:445--466, 1961.

\bibitem{FitzHugh2}
R.~FitzHugh.
\newblock {Mathematical Models of Excitation and Propagation in Nerve}.
\newblock In {\em {Biological Engineering}}. McGrawHill, New York, 1969.

\bibitem{Gyongy99}
I.~Gy{\"o}ngy.
\newblock {Lattice Approximations for Stochastic Quasi-Linear Parabolic Partial
  Differential Equations Driven by Space-Time White Noise II}.
\newblock {\em Potential Anal.}, 11(1):1--37, 1999.

\bibitem{GyongyMillet05}
I.~Gy{\"o}ngy and A.~Millet.
\newblock {On Discretization Schemes for Stochastic Evolution Equations}.
\newblock {\em Potential Anal.}, 23(2):99--134, 2005.

\bibitem{GyongyMillet09}
I.~Gy{\"o}ngy and A.~Millet.
\newblock {Rate of Convergence of Space Time Approximations for Stochastic
  Evolution Equations}.
\newblock {\em Potential Anal.}, 30(1):29--64, 2009.

\bibitem{Hausenblas02}
E.~Hausenblas.
\newblock {Numerical Analysis of Semilinear Stochastic Evolution Equations in Banach Spaces.}
\newblock {\em J. Comput. Appl. Math.} 147:485--516, 2002. 

\bibitem{HodgkinHuxley}
A.~L. Hodgkin and A.~F. Huxley.
\newblock {A Quantitative Description of Membrane Current and its Application
  to Conduction and Excitation in Nerve}.
\newblock {\em J. Physiol.}, 117:500--544, 1952.

\bibitem{HJK2011}
M.~Hutzenthaler, A.~Jentzen and P.~E. Kloeden.
\newblock {Strong and Weak Divergence in Finite Time of Euler's Method for
  Stochastic Differential Equations With Non-Globally Lipschitz Continuous
  Coefficients}.
\newblock {\em Proc. R. Soc. A}, 467(2130):1563--1576, 2011.

\bibitem{HJK2012}
M.~Hutzenthaler, A.~Jentzen and P.~E. Kloeden.
\newblock {Strong Convergence of an Explicit Numerical Method for SDEs With
  Nonglobally Lipschitz Continuous Coefficients}.
\newblock {\em Ann. Appl. Probab.}, 22(4):1611--1641, 2012.

\bibitem{Jentzen}
A.~Jentzen.
\newblock {Pathwise Numerical Approximation of SPDEs with Additive Noise under
  Non-Global Lipschitz Coefficients}.
\newblock {\em Potential Anal.}, 31(4):375--404, 2009.

\bibitem{Keener}
J.~Keener and J.~Sneyd.
\newblock {\em {Mathematical Physiology: I: Cellular Physiology}}, volume~1.
\newblock Springer, 2008.

\bibitem{LordRougement}
G.~J.~Lord and J.~Rougement.
\newblock {A Numerical Scheme for Stochastic PDEs with Gevrey Regularity.}
\newblock {\em IMA J. Numer. Anal.} 24:587--604, 2004.

\bibitem{GinzburgLandau}
D.~Liu.
\newblock {Convergence of the Spectral Method for Stochastic Ginzburg-Landau
  Equation Driven by Space-Time White Noise}.
\newblock {\em Commun. Math. Sci.}, 1(2):361--375, 2003.

\bibitem{LiuLocallyMonotone}
W.~Liu and M.~R{\"o}ckner.
\newblock {SPDE in Hilbert Space with Locally Monotone Coefficients}.
\newblock {\em J. Funct. Anal.}, 259(11):2902--2922, 2010.

\bibitem{Pettersson05}
R.~Pettersson and M.~Signahl.
\newblock {Numerical Approximation for a White Noise Driven SPDE with Locally
  Bounded Drift}.
\newblock {\em Potential Anal.}, 22(4):375--393, 2005.

\bibitem{prevot}
C.~Pr{\'e}v{\^o}t and M.~R{\"o}ckner.
\newblock {\em {A Concise Course on Stochastic Partial Differential
  Equations}}, volume 1905 of {\em Lecture Notes in Mathematics}.
\newblock Springer, Berlin, 2007.

\bibitem{Shardlow99}
T.~Shardlow.
\newblock {Numerical Methods for Stochastic Parabolic PDEs}.
\newblock {\em Numer. Funct. Anal. Optim.} 20(1\&2):121--145, 1999.

\bibitem{Shardlow05}
T.~Shardlow.
\newblock {Numerical Simulation of Stochastic PDEs for Excitable Media}.
\newblock {\em J. Comput. Appl. Math.} 175:429--446, 2005.


\bibitem{TriebelAlt}
H.~Triebel.
\newblock {\em {Interpolation Theory, Function Spaces, Differential
  Operators}}, volume~18 of {\em North-Holland Mathematical Library}.
\newblock North-Holland Publishing Co., Amsterdam, 1978.

\bibitem{TuckwellFHN}
H.~C. Tuckwell.
\newblock {Analytical and Simulation Results for the Stochastic Spatial
  FitzHugh-Nagumo Model Neuron}.
\newblock {\em Neural Comput.}, 20(12):3003--3033, 2008.
\end{thebibliography}
\end{document}